\documentclass[12pt, letterpaper]{amsart}

\usepackage{amsmath,amssymb,latexsym, color
}
\usepackage[dvips,centering,includehead,width=13.1cm,height=22.7cm,
paperheight=23.8cm,paperwidth=17cm
]{geometry}
\usepackage{enumitem}
\usepackage{hyperref}
\usepackage{pict2e}
\usepackage{epic,graphicx}
\usepackage{rotating}
\usepackage{mathtools}
\usepackage{tikz-cd}
\usepackage{ccaption}
\usepackage[commentmarkup=footnote]{changes}
\usepackage{subfiles}
\usepackage{nicefrac}

\changecaptionwidth
\captionwidth{5.6in}

\renewcommand{\ref}{\hyperref}

\def\K{{\mathbb K}}
\def\Z{{\mathbb Z}}
\def\R{{\mathbb R}}

\def\NN{{\mathbb N}}
\def\C{{\mathbb C}}

\def\PP{{\mathbb P}}

\newcommand{\bp}{\boldsymbol{p}}

\newcommand{\bc}{\boldsymbol{c}}

\newcommand{\bn}{\boldsymbol{n}}

\newcommand{\uz}{\underline{z}}
\newcommand{\ux}{\underline{x}}
\newcommand{\oa}{\overline{a}}

\DeclareMathOperator{\Int}{Int}
\DeclareMathOperator{\conj}{conj}\DeclareMathOperator{\ini}{ini}
\DeclareMathOperator{\Ini}{Ini}

\DeclareMathOperator{\val}{Val}
\DeclareMathOperator{\conv}{conv}

\DeclareMathOperator{\Tor}{Tor}\DeclareMathOperator{\gen}{ge}

\DeclareMathOperator{\ord}{ord}
\DeclareMathOperator{\wt}{wt}
\DeclareMathOperator{\OG}{OG}

\DeclareMathOperator{\Spec}{Spec}

\newtheorem{theorem}{Theorem}[section]
\newtheorem{proposition}[theorem]{Proposition}

\newtheorem{corollary}[theorem]{Corollary}
\newtheorem{lemma}[theorem]{Lemma}

\theoremstyle{definition}

\newtheorem{remark}[theorem]{Remark}

\newtheorem{example}[theorem]{Example}
\newtheorem{definition}[theorem]{Definition}

\newtheorem{Construction}{Construction}

\numberwithin{equation}{section}

\definechangesauthor[name={Uriel}, color=purple]{uriel}

\usetikzlibrary{arrows.meta, decorations.markings}

\usepackage{environ}
\makeatletter
\newsavebox{\measure@tikzpicture}
\NewEnviron{scaletikzpicturetowidth}[1]{%
	\def\tikz@width{#1}%
	\begin{lrbox}{\measure@tikzpicture}%
		\BODY
	\end{lrbox}%
	\pgfmathparse{#1/\wd\measure@tikzpicture}%
	\BODY
}
\makeatother

\usepackage{todonotes}
\makeatletter
\providecommand\@dotsep{5}
\makeatother

\newcommand{\bs}[1]{\boldsymbol{#1}}
\newcommand{\mbb}[1]{\mathbb{#1}}

\newcommand\moduli[1]{\overline{\mathcal{M}}^{#1}}
\newcommand\openmoduli[1]{\mathcal{M}^{#1}}
\newcommand\modulispace{\overline{\mathcal{M}}^{\text{trop}}_{g, \langle u_i \rangle}(\Delta)}
\newcommand{\troppoint}{{\scriptstyle \Pi}}
\newcommand{\abstroppoint}{x}
\newcommand{\embtroppoint}{q}
\newcommand{\absalgpoint}{p}
\newcommand{\embalgpoint}{w}

\newcommand{\ogamma}{\overline{\Gamma}}
\DeclareMathOperator{\trop}{trop}
	
	\DeclarePairedDelimiterX\setc[2]{\{}{\}}{\,#1 \;\delimsize\vert\; #2\,}

\begin{document}
\title{Refined enumerative invariants and mixed Welschinger invariants}

\author{Eugenii Shustin}
\address{School of Mathematical Sciences, Tel Aviv
	University, Tel Aviv 69978,
	Israel}
\email{shustin@tauex.tau.ac.il}

\author{Uriel Sinichkin}
\address{School of Mathematical Sciences, Tel Aviv
	University, Tel Aviv 69978,
	Israel}
\email{sinichkin@mail.tau.ac.il}


\date{
	\today
}

\begin{abstract}
	For real toric surfaces and conjugation invariant point conditions with all conjugate pairs on the boundary divisors, we prove that the signed count of real curves of arbitrary genus in the linear system through the given points is invariant under variation of the points, provided the reduced tropicalization is in general position.
	The proof is based on a new relative refined tropical invariant, which is invariant under variation of the point conditions and specializes at $y\to -1$ to this signed count; at $y\to 1$ the same invariant recovers the count of complex curves with prescribed tangency to the boundary.
	We extend the invariant to allow arbitrary tangency orders along the boundary and identify its $y\to 1$ limit with the corresponding complex count.
	Finally, we show that in positive genus the signed real count is not invariant when conjugate pairs are allowed in the interior, even under strong genericity assumptions.
\end{abstract}



\maketitle

\tableofcontents

\section{Introduction}

\subsection{Motivation}
Enumerative geometry studies numerical curve counts in algebraic and symplectic varieties, and over the complex numbers such counts are naturally packaged into deformation-invariant Gromov--Witten type invariants \cite{Kontsevich1998}.
Over the reals, in contrast, the naive number of real curves satisfying incidence conditions typically depends on the chosen configuration of constraints \cite{IKS_enum_welshinger03,Welschinger2003}.

A fundamental way to restore invariance in genus \(0\) is due to Welschinger: one counts real rational curves with signs (now called \emph{Welschinger signs}) determined by the parity of the number of solitary real nodes, obtaining deformation invariants of real symplectic \(4\)-manifolds and lower bounds for the corresponding real enumerative problems \cite{Welschinger2003, Welschinger2005}.
In the algebraic setting of real toric Del Pezzo surfaces, these invariants and their computation were further developed in \cite{IKS_welshinger_caporaso_harris09,IKS_welschinger_alg17}.
Already in low-genus extensions one encounters a basic obstruction: the signed real count in arbitrary positive genus is, in general, not invariant under variation of constraints \cite[Section~3]{IKS_enum_welshinger03} (except in very special situations \cite{itenberg_shustin_real_relative_invariants_2024, shustin_higher_genus_welschinger_invariants_2014}).

Tropical geometry provides an effective combinatorial bridge to both complex and real curve counts on toric surfaces \cite{Mikhalkin2003,ItenbergMikhalkinShustin2007}.
Mikhalkin's correspondence theorem expresses complex enumerative invariants in terms of counts of planar tropical curves with appropriate multiplicities \cite{Mikhalkin2003}.
For the real side, the results of \cite{Mikhalkin2003} yield a computation of the Welschinger invariants for totally real point conditions.
Strikingly, the resulting signed tropical count remains stable for constraints whose tropicalizations stay in tropically general position (we refer to this as being \emph{invariant close to the tropical limit}) even in positive genus.

The tropical computation of the Welschinger invariants (in genus $0$) for non-totally real point conditions was further developed in \cite{Shustin_welschinger_trop06} and was shown to be equal to the count of tropical curves with marked vertices in \cite{GMS_broccoli_intro13}.

In a different direction, Block and G\"ottsche introduced a \(y\)-refinement of tropical curve multiplicities leading to refined curve counts that interpolate between the classical (complex) multiplicity and its real counterpart \cite{BG,itenberg2013block}.
Further developments on the tropical side include broccoli invariants and refined elliptic tropical invariants \cite{Gottsche_shroeter_refined_broccoli,Schroeter_shustin_refined_elliptic08}.

In \cite{CNRBI} we defined and studied a refined tropical count for tropical curves with marked vertices.
The results of \cite{CNRBI} have two main limitations: first, the refined count is not invariant in certain settings; second, the limit of the count as $y\to -1$ has no clear geometric interpretation.

\subsection{Main results}

In the present work, we address the limitations of \cite{CNRBI} by placing high-multiplicity points on the boundary of the toric surface.

Our central result is the invariance of the real count close to the tropical limit in every genus if all conjugate pairs of points are located on the boundary divisors.
\begin{theorem}\label{thm:real_invariance}
    Fix a non-negative integer $g\ge 0$ and let $X=\Tor_\R(P)$ be a toric surface corresponding to a convex lattice polygon $P\subset\R^2_\Z$.
    Let $\bs{\embalgpoint}\subset X$ be a conjugation invariant set of $|\partial P|+g-1$ points such that all conjugate pairs of points are located on the boundary divisors.
    If the reduced tropicalization of $\bs{\embalgpoint}$ is in general position, then the signed count of real curves in the linear system $|{\mathcal L}_P|$ passing through $\bs{\embalgpoint}$ depends only on $P$ and the number of conjugate pairs of points on each component of the boundary divisor, and does not depend on the choice of $\bs{\embalgpoint}$.
\end{theorem}

In fact, we have defined a refined count that is invariant under variation of the points $\bs{\embalgpoint}$ as long as their reduced tropicalization is in general position (see Proposition \ref{prop:refined_weight_invariant}) and whose value at $y\to -1$ is equal to the desired signed count (see Proposition \ref{prop:refined_invariant_real} and Theorem \ref{thm:real_final_count}), which proves Theorem \ref{thm:real_invariance}.
Investigating the limit of the refined invariant as $y\to 1$ we obtain the number of algebraic curves with prescribed tangency points to the boundary divisors, see Theorem \ref{thm:complex_tangency_num} and Proposition \ref{prop:refined_invariant_complex} below.

Moreover, we have defined the relative refined invariant allowing arbitrary multiplicity at the boundary divisors.
For this extended invariant, we can show that the limit as $y\to 1$ is equal to the number of algebraic curves with prescribed tangency points of the prescribed orders to the boundary divisors, see Theorem \ref{thm:complex_tangency_num} and Proposition \ref{prop:refined_invariant_complex} below.
Of course the real enumerative question does not make sense for the extended invariant.

\subsection{Structure of the paper}
The paper is organized as follows.
In Section \ref{sec:preliminaries} we fix notations and recall the relevant tropical and non-Archimedean background.
Section \ref{sec:complex_count} proves the complex correspondence theorem.
Section \ref{sec:real_conjugate} develops the real correspondence theorem for essentially peripheral constraints.
Section \ref{sec:refined} introduces the relative refined invariant and proves its invariance and limit at $y\to \pm 1$.
We also discuss the computation of the relative refined invariant in the case of the projective plane, and when all boundary points lie on the same boundary component, in Section \ref{sec:computation}.
Finally, in Section \ref{sec:negative_result} we show that outside the essentially peripheral regime one cannot expect invariance of the signed real count in positive genus, even under the most restrictive conditions.

\subsection*{Acknowledgments}
The authors were supported by the BSF grant no.  2022/157, and by the Bauer-Neuman Chair in Real and Complex Geometry.
The second author was also supported by the ISF grant no. 1918/23, and by the Milner Foundation.

\section{Preliminaries}\label{sec:preliminaries}

\subsection{Notations}\label{sec:notations}
{\bf(1)} We use the complex field $\C$, the real field $\R$, the field $\K$ of locally convergent Puiseux series with complex coefficients, and the field $\K_\R\subset \K$ of locally convergent Puiseux series with real coefficients.
Note that $\K$ is an algebraically closed field, it has a natural conjugation $\bc\colon \K\to \K$, and $\K_\R$ is the fixed field of $\bc$.
For an element $a=\sum_{r\ge r_0}a_rt^r\in\K^*$, $a_{r_0}\ne 0$, denote
\begin{equation*}\trop(a)=r_0,\quad \ini(a)=a_{r_0},\quad \Ini(a)=a_{r_0}t^{r_0}\ .\end{equation*}
For $p\in (\K^\times)^2$ denote by $\trop(p)\in \R^2$ the point $(\trop(p_1), \trop(p_2))$.
Let $F=\sum_{\omega\in P\cap\Z^n}a_\omega\uz^\omega\in\K[\uz^{\pm 1}]$, $\uz=(z_1,...,z_n)$, have Newton polytope $P$. It
yields a {\it tropical polynomial}
$$N(\ux)=N_F(\ux)=\max_{\omega\in P\cap\Z^n}(\langle\omega,\ux\rangle-\trop(a_\omega)),\quad N:\R^n\to\R\ ,$$
and its Legendre dual, {\it valuation function} $\nu=\nu_N:P\to\R$, whose graph defines a subdivision $\Sigma_\nu$
of $P$ into linearity domains which all are convex lattice polytopes. One can write
$$F(\uz)=\sum_{\omega\in \Delta\cap\Z^n}(a^0_\omega+O(t^{>0}))t^{\nu(\omega)}\uz^\omega\ ,$$ where
$a^0_\omega\in\C$, $a^0_\omega\ne0$ for all $\omega$ vertices of the subdivision $\Sigma_\nu$. Given a face $\delta$ of
the subdivision $\Sigma_\nu$, we write
$$F^\delta(\uz)=\sum_{\omega\in\delta\cap\Z^n}a_\omega\uz^\omega,\quad\ini(F^\delta)(\uz)=
\sum_{\omega\in\delta\cap\Z^n}a^0_\omega\uz^\omega\in\C[\uz]\ ,$$
$$\Ini(F^\delta)(\uz)=\sum_{\omega\in\delta\cap\Z^n}a^0_\omega t^{\nu(\omega)}\uz^\omega
\in\K[\uz]\ .$$

\smallskip

{\bf(2)} All lattice polyhedra we consider lie in Euclidean spaces $\R^2$ with fixed integral lattices
$\Z^2\subset\R^2$.

For a convex lattice polygon $P\subset\R^2$, we denote
by $\Tor(P)$ the complex toric surface associated with $P$
and by $\Tor^*(P)\subset\Tor(P)$ the big torus (the dense orbit of the torus action). Next, denote by ${\mathcal L}_P$ the tautological line bundle over $\Tor(P)$, by
$|{\mathcal L}_P|$ the linear system generated by the non-zero global sections (equivalently, by the monomials
$x^iy^j$, $(i,j)\in P\cap\Z^2$). We also use the notation ${\mathcal L}_P(-Z)$ for
${\mathcal L}_P\otimes{\mathcal J}_{Z/\Tor(P)}$, where ${\mathcal J}_{Z/\Tor(P)}$ is the ideal sheaf of a zero-dimensional subscheme $Z\subset\Tor(P)$.

Let $\bn:\widehat C\to X$ be a non-constant morphism of a complete smooth irreducible curve $\widehat C$ to a toric surface $X$. We call it {\it peripherally unibranch} if, for any
toric divisor $D\subset X$, the divisor $\bn^*(D)\subset\widehat C$ is concentrated at one point.
Respectively we call the curve $C=\bn_*\widehat C$ {\it peripherally smooth},
if it is smooth at its intersection points with each toric divisor.

\smallskip
{\bf(3)}  For a vector $v\in\Z^n$, resp. a lattice segment $\sigma\subset\R^n_\Z$, we denote
its lattice length by $\|v\|_\Z$,
resp. $\|\sigma\|_\Z$. More generally, for an $m$-dimensional lattice polytope
$P\subset\R^n_\Z$, $n\ge m$, we denote by $\|P\|_\Z$ its $m$-dimensional lattice volume (i.e., the ratio
of the Euclidean $m$-dimensional volume of $P$ and the minimal
Euclidean volume of a lattice simplex inside the affine $m$-dimensional subspace of $\R^n_\Z$ spanned by $P$).

For a graph $G$, denote by $G^0$ the set of vertices and by $G^1$ the set of edges of $G$.

\subsection{Tropical curves}\label{sec:trop_curves}

A {\it plane tropical curve} is a tuple $(\ogamma, h)$, where
\begin{itemize}
    \item $\ogamma$ is a finite connected compact topological graph without bivalent vertices with a metric on $\Gamma:=\ogamma\setminus \ogamma_{\infty}^0$, where $\ogamma_{\infty}^0$ is the set of univalent vertices of $\ogamma$.
    The metric on $\Gamma$ should assign finite length to compact edges of $\Gamma$, also called bounded edges, and infinite length to non-compact edges of $\Gamma$.
	We denote by $\Gamma^1_b$ the set of bounded edges and by $\Gamma^1_\infty$ the set of unbounded edges.
    
    \item $h:\Gamma\to\R^2$ is a continuous map such that $h$ is affine-integral on each edge of $\Gamma$ in the length coordinate, and it is nonconstant on at least one edge of $\Gamma$.
    Furthermore, at each vertex $V$ of $\Gamma$, the balancing condition holds
    $$\sum_{E\in\Gamma^1,\ V\in E}
    \oa_V(E)=0\ ,$$ where $\oa_V(E)$ is the image under the differential $D(h\big|_E)$
    of the unit tangent vector to $E$ emanating from
    its endpoint $V$. We call $\oa_V(E)$ the {\it directing vector} of $E$ (centered at $V$).   
	
    For univalent vertex $V\in \ogamma^0_\infty$, we denote by $\oa_V$ the directing vector of the edge $E$ incident to $V$, and by $h(V):=[h(E)]\in \nicefrac{\R^2}{\oa_V\R}$ the line containing the image of $E$.
\end{itemize}

Define the {\it genus} of $(\ogamma, h )$ as the first Betti number of $\ogamma$.

When this causes no ambiguity, we will denote the tuple $(\ogamma, h)$ by $\Gamma$.

The \emph{valency} of a vertex $ V\in \Gamma^0 $ is the number of edges incident to it.
A vertex $ V\in \Gamma $ is called {\it flat} or {\it collinear} if the linear span of $ {\setc{\oa_V(E)}{E\in \Gamma^1, V\in E}} $ is one dimensional.
Given a non-collinear vertex $V\in\Gamma^0$, the directing vectors $\oa_V(E)$ over all edges $E$ incident to $V$, being positively rotated by $\frac{\pi}{2}$ form a nondegenerate convex lattice polygon
$P(V)$, which we call {\it dual} to $V$.
The multi-set $\Delta=\Delta(\Gamma,h)
\subset\R^2_\Z$ of vectors $\{\oa_V(E)\ne0\ : \ E\in\Gamma^1_\infty,
\ V\in\Gamma^0,\ V\in E\}$ is called the {\it degree} of $(\Gamma,h)$. It is easy to see that $\Delta$
is nonempty and {\it balanced} (i.e., the vectors of $\Delta$ sum up
to zero. The degree $\Delta$ is called {\it primitive} if it consists of primitive vectors (i.e., vectors of lattice length $1$).
Positively rotated by $\frac{\pi}{2}$, the vectors of $\Delta$
can be combined into a convex lattice polygon $P=P(\Gamma,h)$, called
the {\it Newton polygon} of $(\Gamma,h)$. In this case, we say that the degree $\Delta$ {\it induces} the polygon $P$. The degree $\Delta$ is called {\it nondegenerate} if $\dim P(\Gamma,h)=2$.

To each edge $E\in\Gamma^1$ we assign the {\it weight} $$\wt(E)=\big\|\oa_V(E)\big\|_\Z\ .$$

A {\it marked plane tropical curve} is a tuple $(\Gamma,h, \boldsymbol{\abstroppoint})$, where $ (\Gamma,h) $ is a plane tropical curve and $\boldsymbol{\abstroppoint}$ is an ordered subset of points of $\ogamma$.
An end $E\in \Gamma^1_\infty$ is called {\it fixed} if its univalent end is marked.

Two marked tropical curves $(\Gamma, h, \boldsymbol{\abstroppoint})$ and $(\Gamma', h', \boldsymbol{\abstroppoint'})$ are said to be \emph{isomorphic} if there exists an isomorphism of graphs $\varphi:\ogamma\to \ogamma'$ that restricts to an isomorphism of metric graphs $\varphi|_{\Gamma}:\Gamma\to\Gamma'$ with $h'\circ\varphi|_{\Gamma}=h$ and $\varphi(\boldsymbol{\abstroppoint}) = \boldsymbol{\abstroppoint'}$.

\begin{definition}
	Let $(\Gamma,h)$ be a plane tropical curve, and let $V\in\Gamma^0$ be a trivalent vertex, which is not incident to
	contracted edges.
	Define the Mikhalkin multiplicity of $V$ by
	$$\mu(V)=\big|\oa_V(E_1)\wedge\oa_V(E_2)\big|\ ,$$
	where $E_1,E_2\in\Gamma^1$ are some two incident to $V$ edges.
\end{definition}

\begin{definition}\label{def:regular_trop_curve}
    A marked tropical curve is called \emph{regular} if every connected component of $\ogamma\setminus \boldsymbol{\abstroppoint}$ is simply connected and contains exactly one univalent vertex.
    If $(\ogamma, h, \bs{\abstroppoint})$ is regular, the edges of $\Gamma\setminus \bs{\abstroppoint}$ can be oriented in a unique way for which the unique unfixed unbounded edge oriented towards infinity and every unmarked vertex has exactly one outgoing edge.
    We call such orientation \emph{the regular orientation on $\Gamma$}.
\end{definition}

\begin{definition}
	Let $ (\Gamma, h, \boldsymbol{\abstroppoint}) $ be a tropical curve.
	For every marked point $ \abstroppoint_i\in \boldsymbol{\abstroppoint} $ denote by $ { \Omega}_i \in \ogamma^1\cup \ogamma^0 $ the unique cell of $ \ogamma $ containing $ \abstroppoint_i $ in its relative interior.
	Additionally, denote by $\underline{\Gamma}$ the abstract graph obtained from $\ogamma$ by forgetting the metric structure.
	We define the \emph{combinatorial type} (or simply \emph{type}) of $ (\Gamma, h, \boldsymbol{\abstroppoint}) $ to be the tuple $ (\underline{\Gamma}, (\wt(E))_{E\in \Gamma^1}, \bs{{\Omega}}, (\oa_V(E))_{V\in E\in \Gamma^1}) $.
	We will denote the combinatorial type of $ \Gamma $ by $ [\Gamma] $.
\end{definition}

\subsection{Moduli space of plane tropical curves and evaluation}\label{sec:moduli_and_evaluation}

\begin{definition}\label{def:trop_point_conditions}
    Tropical point conditions are a sequence of pairs $(\embtroppoint_i, u_i)$ where $u_i\in \Z^2$ and $\embtroppoint_i\in \nicefrac{\R^2}{u_i\R}$.

    A marked plane tropical curve $(\Gamma, h, \bs{\abstroppoint})$ satisfies tropical point conditions $\langle (\Lambda_i, u_i) \rangle_{i=1}^n$ if for every $i$, $h(\abstroppoint_i)= \embtroppoint_i$ and if $u_i\ne 0$ then $\abstroppoint_i$ is adjacent to an unbounded end $E$ of weight at least $|u_i|$ and with $u_i\in \oa_E(E)\R_{>0}$ (i.e. the direction of the end is the same as the direction of $u_i$).
    
    The type of tropical point conditions $\langle (\embtroppoint_i, u_i) \rangle_{i=1}^n$ is the sequence $\langle u_i \rangle_{i=1}^n$.

	A degree $\Delta$ \emph{admits} a type of tropical point conditions if the multiset of nonzero vectors $u_i$ in the type of the tropical point conditions is a subset of $\Delta$.
	If a tropical curve satisfies tropical point conditions of a given type, then its degree admits that type of tropical point conditions.
\end{definition}

We fix a nondegenerate balanced multiset $\Delta \subseteq \mbb{Z}^2\setminus \{0\}$ that induces a polygon $P$, a nonnegative integer $g\le \left|\Int(P)\cap \Z^2 \right|$, and a type of tropical point conditions $\langle u_i \rangle_{i=1}^n$ that is admitted by $\Delta$.

Given a combinatorial type $ \alpha = (\Gamma, \bs{w}, \bs{\Omega}, (\oa_V(E))) $ of degree $ \Delta $, genus $ g $ such that $\Gamma$ satisfies a tropical point conditions of type $\langle u_i \rangle_{i=1}^n$, we can parameterize the tropical curves of type $ \alpha $ by a polyhedron $ \openmoduli{\alpha} \subseteq \mbb{R}^{|\Gamma^1|-|\Delta|+|\bs{c}\cap \Gamma^1|+2} $.
Indeed, every such tropical curve is defined by the length of its bounded edges, the distance between vertices and neighboring marked points and a position of an arbitrarily chosen vertex (and those lengths and distances should satisfy linear equations and inequalities in order to represent a tropical curve).

The points in the closure ${\moduli{\alpha}} \subseteq \mbb{R}^{|\Gamma^1|-|\Delta|+|\bs{c}\cap \Gamma^1|+2}$ correspond to tropical curves of combinatorial types attained from $\alpha$ by possibly contracting some edges and moving marked points from an edge to one of its boundary vertices.
Thus, for such combinatorial types $ \beta $ we get that $ {\moduli{\beta}} $ is naturally a face of $ {\moduli{\alpha}} $.
We will say that $\beta$ is a \emph{degeneration} of $\alpha$, and $\alpha$ is a \emph{regeneration} of $\beta$.
Gluing the spaces ${\moduli{\alpha}}$ via this identification, we obtain the \emph{moduli space of genus $g$ tropical curves satisfying tropical point conditions of type $\langle u_i \rangle_{i=1}^n$}:
\[ \modulispace := \lim_{\stackrel{\longrightarrow}{\alpha}} \moduli{\alpha}. \]

The following lemma is an immediate consequence of \cite[Proposition 3.9]{Gathmann2007}.
\begin{lemma}\label{lemma:dim_moduli_space}
	Let $\alpha$ be a combinatorial type of a tropical curve $\Gamma$ satisfying tropical point conditions of type $\langle u_i \rangle_{i=1}^n$.
	Then $\openmoduli{\alpha}$ is a polyhedron of dimension
	\[ \dim \openmoduli{\alpha} \le |\Delta|+g-1+|\setc{i}{u_i=0}| \]
with equality if and only if $\Gamma$ is trivalent and all marked points $\abstroppoint_i$ with $u_i=0$ are disjoint from each other and from the vertices of $\Gamma$.
\end{lemma}

\begin{corollary}
	The dimension of $ \modulispace $ is $ |\Delta|+g-1+|\setc{i}{u_i=0}| $.
\end{corollary}

We conclude this section by defining the evaluation map and analyzing its behavior in codimension zero and one.

\begin{definition}\label{def:evaluation_map}
	The \emph{evaluation map} is the map given by
	\begin{align*}
		Ev: \modulispace & \to \prod_{i=1}^n \nicefrac{\R^2}{u_i\R} \\
		(\Gamma, h, \bs{\abstroppoint}) & \mapsto \left( h(\abstroppoint_1), \dots, h(\abstroppoint_n) \right)
	\end{align*}
\end{definition}

The proof of the following lemma is immediate, and omitted for conciseness.
\begin{lemma}\label{lemma:Ev_injective}
For a tropical curve $(\Gamma,h, \bs{\abstroppoint})$ of combinatorial type $\alpha$, the restriction $\text{Ev}|_{\openmoduli{\alpha}}$ is injective if and only if $\Gamma\setminus \bs{\abstroppoint}$ contains neither components of positive genus nor components with more than one unbounded unmarked edge.
\end{lemma}

From now on, set $ n = |\Delta|+g-1 $, so that $Ev$ is a map of polyhedral complexes of the same dimension.
We say that a combinatorial type $\alpha$ is \emph{enumeratively essential} if there exists a combinatorial type $\beta$ with $\openmoduli{\alpha}\subseteq \moduli{\beta}$ and the image of $\openmoduli{\beta}$ under $\text{Ev}$ has maximal dimension.

In this setting Lemma \ref{lemma:dim_moduli_space} and Lemma \ref{lemma:Ev_injective} imply the following characterization of the generic case of the evaluation map.
\begin{lemma}\label{lemma:Ev_generic}
	Let $\openmoduli{\text{generic}} \subseteq \modulispace$ be the set of trivalent regular genus $g$ tropical curves in which all marked points $\abstroppoint_i$ with $u_i=0$ are disjoint from each other and from the vertices of $\Gamma$.
	Then $Ev(\modulispace\setminus \openmoduli{\text{generic}})$ is of positive codimension in $ \prod_{i=1}^n \nicefrac{\R^2}{u_i\R} $.
\end{lemma}

The codimension $1$ case is a straightforward consequence of \cite[Proposition 3.9]{Gathmann2007}.
\begin{lemma}\label{lemma:Ev_codim1}
	Let $\langle u_i \rangle_{i=1}^n$ be a type of tropical point conditions which is admitted by $\Delta$ and suppose that $ n = |\Delta|+g-1 $.
	Let $\Gamma\in \modulispace$ be a tropical curve of the enumeratively essential combinatorial type $\alpha$ and suppose that $\dim \openmoduli{\alpha} = n + |\setc{i}{u_i=0}| - 1$ (i.e. of codimension $1$ in $ \modulispace $).
	Then $\Gamma$ is of genus $g$ and one of the following holds:
	\begin{enumerate}
		\item Either $\Gamma$ is regular, has one $4$-valent marked vertex and all the other vertices of $\Gamma$ are trivalent and all the marked points $\abstroppoint_i$ with $u_i=0$ are disjoint from each other and from the vertices of $\Gamma$;
		\item or $\Gamma$ is trivalent, all marked points $\abstroppoint_i$ with $u_i=0$ are disjoint from each other and from the vertices of $\Gamma$ except for a unique marked vertex $V$, and all components of $\Gamma\setminus \bs{\abstroppoint}$ are unbounded except for a unique component that contains a single edge attached to $V$;
		\item or $\Gamma$ is regular, has two $4$-valent marked vertices that are joined by a collinear cycle, all other vertices of $\Gamma$ are trivalent and all the marked points $\troppoint_i$ with $u_i=0$ are disjoint from each other and from the vertices of $\Gamma$;
	\end{enumerate}
\end{lemma}

\begin{proof}
	By \cite[Proposition 3.9]{Gathmann2007}, we get that $|\bs{\abstroppoint}\cap\Gamma^0| + g-\gen(\Gamma) \le 1$.
	Consider first the case where $\gen(\Gamma) = g-1$ and $|\bs{\troppoint}\cap\Gamma^0| = 0$.
	Since the restriction of an injective map is injective, we get by Lemma \ref{lemma:Ev_injective} that all the components of $\Gamma\setminus \bs{\abstroppoint}$ are of genus $0$ and contain at most one unbounded edge.
	Noting that since $|\bs{\abstroppoint}\cap\Gamma^0| = 0$ all the marked points $\abstroppoint_i$ with $u_i=0$ are in the interior of the edges of $\Gamma$, by a standard Euler characteristic calculation we get that $\Gamma\setminus \bs{\abstroppoint}$ contains a unique bounded component.
	For a combinatorial type $\beta$ degenerating to $\alpha$ and having $Ev(\openmoduli{\beta}) $ of maximal dimension and a tropical curve $\Gamma'$ attaining $\beta$, we get that the connected components of $\Gamma'\setminus \bs{\troppoint}$ are in bijection with the connected components of $\Gamma\setminus \bs{\troppoint}$ meaning that $\Gamma'$ is not regular in contradiction with Lemma \ref{lemma:Ev_generic}.

	The case where $\gen(\Gamma) = g$ and $|\bs{\abstroppoint}\cap\Gamma^0| = 1$ give rise to option $(2)$ in the statement of the lemma.
	
	If $\gen(\Gamma) = g$ and $|\bs{\abstroppoint}\cap\Gamma^0| = 0$, then a similar Euler characteristic calculation shows that $\Gamma\setminus \bs{\abstroppoint}$ has no bounded components, so we get options $(1)$ and $(3)$ in the statement of the lemma by \cite[Proposition 3.9]{Gathmann2007}.
\end{proof}

\subsection{Tropicalization of algebraic curves over a non-Archimedean field}\label{sec-trop}
In this section, we describe \emph{tropicalization}, i.e., the process of associating a tropical curve to an algebraic curve over $\K$.
In fact, we will see that the tropicalization of an algebraic curve produces a richer structure than just a tropical curve, which we call an \emph{enhanced tropical curve}.
Most of this section is borrowed from \cite{CNRBI}, with the necessary adaptations.

\subsubsection{Enhanced tropical curves}\label{sec:enhanced_trop_curves}
Let $(\Gamma, h, \boldsymbol{\abstroppoint})$ be a marked plane tropical curve, and let $V\in \Gamma^0$ be a non univalent vertex.
A \emph{limit curve} associated to $V$ is a tuple $((p_{V,E})_{V\in E}, \varphi_V)$, where $(p_{V,E})_E\subset \PP_\C^1$ is a collection of distinct points on $\PP_\C^1$ indexed by the adjacent edges to $V$, and $\varphi_V$ is a morphism $\PP_\C^1\setminus \{p_{V,E}\} \to (\C^*)^2 $ with $ \ord_{p_{V,E}}(\varphi_{V, i}) = (\oa_{V,E})_i$. 
For a non-flat vertex $V$, a limit curve associated to a vertex $V$ can be extended to a morphism $\varphi:\PP^1_\C\to Tor_{\C}(P_V)$, where $P_V$ is the polygon dual to $V$.
In particular, if $E\in \Gamma^1$ is an edge with endpoints $V_1$ and $V_2$, then both $\varphi_{V_1}(p_{V_1,E})$ and  $\varphi_{V_2}(p_{V_2,E})$ are contained in the mutual toric boundary $Tor_{\C}(P_{V_1})\cap Tor_{\C}(P_{V_2})$. 
We say that the limit curves are \emph{compatible along $E$} if $\varphi_{V_1}(p_{V_1,E}) = \varphi_{V_2}(p_{V_2,E})$.
An \emph{enhanced plane tropical curve} is a plane tropical curve with compatible (along all edges) limit curves associated to its vertices.

\subsubsection{Parameterized tropical limit}\label{sec-ptl}
Let $\Delta\subset\Z^2\setminus\{0\}$ be a nondegenerate, primitive, balanced multiset which induces a convex lattice polygon $P$, and let $C\in|{\mathcal L}_P|_\K$ be a reduced,
irreducible curve of genus $\gen(C)=g$, which passes through a sequence $\bs{\embalgpoint}$ of $n$ distinct points in $(\K^*)^2\subset\Tor_\K(P)$ and does not hit intersection points of toric divisors. In particular, it can be given by an
equation
$$F(x,y)=\sum_{(i,j)\in P\cap\Z^2}t^{\nu'(i,j)}(a^0_{ij}+O(t^{>0}))x^iy^j\ ,$$
where $a^0_{ij}\in\C$, $(i,j)\in P\cap\Z^2$, and $a^0_{ij}\ne0$ if the coefficient
of $x^iy^j$ in $F$ does not vanish (for instance, when $(i,j)$ is a vertex of $P$).
We then define a convex, piecewise-linear function $\nu:P\to\R$, whose graph is the lower part of
the $\conv\{(i,j,\nu'(i,j)),\ (i,j)\in P\cap \Z^2\}\subset\R^3_\Z$.
Via a parameter change $t\mapsto t^r$, we can make
$\nu(P\cap\Z^2)\subset\Z$. Denote by $\Sigma_\nu$ the subdivision
of $P$ into linearity domains of $\nu$, which are convex lattice polygons $P_1,...,P_m$. We then have
\begin{equation}F(x,y)=\sum_{(i,j)\in P\cap\Z^2}t^{\nu(i,j)}(c^0_{ij}+O(t^{>0}))x^iy^j\ ,\label{e-new60}\end{equation}
where $c^0_{ij}\ne0$ for $(i,j)$ a vertex of some of the $P_1,...,P_m$. This data defines a flat family
$\Phi:{\mathfrak X}\to\C$, where ${\mathfrak X}=\Tor(\OG(\nu))$ and
$$\OG(\nu)=\{(i,j,c)\in\R^3_\Z\ :\ (i,j)\in P,\ c\ge\nu(i,j)\}$$ is the overgraph of $\nu$. The central fiber
${\mathfrak X}_0=\Phi^{-1}(0)$ splits into the union of toric surfaces $\Tor(P_k)$, $1\le k\le m$, while the other fibers are isomorphic to
$\Tor(P)$. The evaluation of the parameter $t$ turns the given curve $C$ into an inscribed family of curves
\begin{equation*}C^{(t)}\subset{\mathfrak X}_t,\quad C^{(t)}\in|{\mathcal L}_{P}|,\quad t\in(\C,0)\setminus\{0\}\ ,\end{equation*} (where $(\C,0)$ always means a sufficiently small disc in $\C$ centered at zero)
which closes up to a flat family over $(\C,0)$ with the central element
$$C^{(0)}=\bigcup_{k=1}^mC^{(0)}_k\ ,$$ where
$$C^{(0)}_k=\left\{F^{(0)}_k(x,y):=\sum_{(i,j)\in P_k\cap\Z^2}c^0_{ij}x^iy^j=0\right\}\in|{\mathcal L}_{P_k}|,
\ 1\le k\le m\ .$$
Whenever $P_i$ and $P_j$ share an edge, the corresponding limit curves $C^{(0)}_i$ and $C^{(0)}_j$ intersect at a point of $\Tor(P_i)\cap \Tor(P_j)$.

Let $\bn:(\widehat C,\bs{\absalgpoint})\to(C,\bs{\embalgpoint})$ be the normalization (where $\bs{\absalgpoint}$ is a sequence of $n$ distinct points on $\widehat C$), or, equivalently, the family
\begin{equation}\bn_t:(\widehat C^{(t)},\bs{\absalgpoint}(t))\to (C^{(t)},\bs{\embalgpoint}(t))\hookrightarrow{\mathfrak X}_t,\quad t\in(\C,0)\setminus\{0\}\ ,
	\label{euc3}\end{equation} where
each $\widehat C^{(t)}$ is a smooth complex curve of genus $\gen(C)=g$ (cf. \cite[Theorem 1, page 73]{Tei} or \cite[Proposition 3.3]{ChL}).

After a suitable untwist $t\mapsto t^r$, the family
(\ref{euc3}) admits a flat extension to $0\in(\C,0)$
with the central element $\bn_0:(\widehat C^{(0)},\bs{\absalgpoint}(0))\to{\mathfrak X}_0$, where $\widehat C^{(0)}$ is a connected nodal curve of arithmetic genus $g$ (see, for instance \cite[Theorem 1.4.1]{AV}), none of whose components is entirely mapped to a toric divisor, $\bp(0)$ is a sequence of $n$ distinct points of $\widehat C^{(0)}$ disjoint from the intersection points of the components of $\widehat C^{(0)}$,
and such that $(\bn_0)_*(\widehat C^{(0)},\bs{\absalgpoint}(0))=(C^{(0)},\bs{\embalgpoint}(0))$.

With the central fiber one can associate an enhanced plane marked
tropical curve $(\Gamma,{\boldsymbol{\abstroppoint}}, h)$ as defined in \cite[Section 2.2.1]{Ty}. In particular (all other details can be found in \cite[Section 2]{Ty}), the vertices of $\Gamma$ bijectively correspond to components of $ C^{(0)}$, the finite edges of $\Gamma$ bijectively correspond to the intersection points of distinct components of $ C^{(0)}$, and the infinite edges of $\Gamma$ correspond to the points of $ C^{(0)}$ mapped to toric divisors $\Tor(e)$ such that $e\subset\partial P$.
Thus, we obtain the
{\it parameterized tropical limit} (briefly, {\it PTL}) of $\quad$ $(\bn:\widehat C\to C,\bs{\embalgpoint})$ to be the enhancement of the plane marked tropical curve $(\Gamma,\bs{\abstroppoint}, h)$ given by the limit curves $\varphi_V: \PP^1 \to \Tor(\delta_V)$ where $\delta_V$ is the polygon dual to $V$.

\subsection{Real tropical curves}\label{sec:real_trop_curves}

\begin{definition}
    A real tropical curve is a tropical curve $h:\Gamma\to \mbb{R}^2$ together with an involution $\bc:\ogamma\to \ogamma$ such that $h\circ \bc = h$.

A marking of a real tropical curve is always assumed to be invariant under the involution $\bc$ as a set.

    A marking $\boldsymbol{\abstroppoint}$ is called peripheral if all the points $\abstroppoint_{i}$ which are not fixed by $\bc$ are univalent vertices of $\ogamma$.

    A \emph{(peripherally) marked real tropical curve} is a real tropical curve together with a (peripheral) marking.

    The \emph{combinatorial type} of a marked real tropical curve is the combinatorial type of the underlying marked tropical curve together with the involution $\bc$ on the underlying graph. 
\end{definition}

For a real tropical curve $(\Gamma, h, \bs{\abstroppoint}, \bc)$, the quotient $(\Gamma/\bc,h/\bc,\boldsymbol{\abstroppoint}/\bc)$ is a well defined plane tropical curve. We can contract the edges, where $h/\bc$ is constant, then remove all bivalent vertices by gluing the two incident edges into one edge.
This operation results in a plane marked tropical curve which we denote $(\Gamma',h',\boldsymbol{\abstroppoint}')$. 
Setting $\Gamma^{\bc}:=Fix(\bc)$, we denote by $\Gamma'_{re}\subset\Gamma'$ the subgraph obtained from $\Gamma^{\bc}$, and put $\Gamma'_{im}=\overline{\Gamma'\setminus\Gamma'_{re}}$ be the closure of the complement in $\Gamma'$.

Suppose that $\bn: (\widehat C, \bs{\absalgpoint}) \to \Tor_{\K_\R}(P)$ is a real curve, meaning that $\widehat C$ possesses a real structure compatible with the real structure on $\Tor_{\K_\R}(P)$, and let $(\Gamma, \bs{\abstroppoint}, h, \{\varphi_V\}_{V\in \Gamma^0})$ be its parameterized tropical limit.
Then the conjugation on $\widehat C$ naturally descends to an involution $\bc:\Gamma\to\Gamma$ with the property that $p_{\bc(V), \bc(E)}=\overline{p_{V, E}}$ for all $V\in \Gamma^0$ and $E\in \Gamma^1$ that are adjacent to $V$ and $\varphi_{\bc(V)} = \varphi_V\circ \conj$ for all $V\in \Gamma^0$ where $\conj$ is the conjugation on $(\C^\times)^2$.
An enhancement of a real tropical curve satisfying those conditions is called conjugation invariant enhancement.

\section{Complex correspondence theorem}\label{sec:complex_count}

In this section we provide a tropical count of genus $g$ complex curves satisfying a given set of algebraic point conditions.

\begin{definition}\label{def:alg_point_conditions}
    Algebraic point condition in $\Tor_\K(P)$ is a pair $(\embalgpoint, m)$ where $\embalgpoint\in \Tor_\K(P)$ and $m\in \NN$, such that if $\embalgpoint\in (\K^\times)^2$, we have $m=1$.

    A curve $C\subset \Tor_\K(P)$ satisfies a sequence of algebraic point conditions $\langle (\embalgpoint_i, m_i) \rangle_{i=1}^n$ if for every $i$, $\embalgpoint_i\in C$ and for every $i$ with $m_i>1$, there is a local branch of $C$ intersecting the toric boundary divisor at $\embalgpoint_i$ with multiplicity at least $m_i$.

    The type of algebraic point conditions $\langle (\embalgpoint_i, m_i) \rangle_{i=1}^n$ is the sequence of pairs $(S, m_i)$ where $S$ is the stratum of $\Tor_\K(P)$ containing $\embalgpoint_i$.
\end{definition}

Before we proceed further, let us make the following technical assumption intended to rule out the study of unnecessary situations. Namely, we assume that for every $i$ with $\embalgpoint_i\in (\K^\times)^2$, we have
\begin{equation}\begin{cases}\embalgpoint_i=(t^{a_i}(\eta_i+O(t^M)),t^{b_i}(\zeta_i+O(t^M))),&\\ \qquad a_i,b_i\in\mathbb{Q},\ \eta_i,\zeta_i\in\R^*,\quad i=1,...,n,&\end{cases}\label{e5}\end{equation}
where $\eta_i,\zeta_i$ are chosen generically, and $M>0$ is sufficiently large.

\begin{definition}\label{def:trop_point_conditions_tropicalization}
    The tropicalization of algebraic point conditions $\langle (\embalgpoint_i, m_i) \rangle_{i=1}^n$ is defined as the tropical point conditions (see Definition \ref{def:trop_point_conditions}) $\langle (\trop(\embalgpoint_i)+u_i\R, m_i\cdot u_i) \rangle_{i=1}^n$, where $u_i$ is the primitive integral vector corresponding to the boundary divisor at $\embalgpoint_i$ (setting $u_i=0$ if $\embalgpoint_i\in (\K^\times)^2$).
\end{definition}

\begin{example}
    Let $P$ be the triangle with vertices $(0,0)$, $(2,0)$, $(0,1)$ and denote $\sigma:=\conv\{(0,0), (2,0)\}$.
    Consider the algebraic point conditions $\langle (\embalgpoint_1, 1), (\embalgpoint_2, 2) \rangle$ where $\embalgpoint_1=(1,1)\in (\K^\times)^2$ and $\embalgpoint_2=(1,0)\in \Tor_\K(\sigma)$.
    Then the condition on algebraic curves to pass through $\embalgpoint_1$ is translated to the condition on their tropicalization to pass through $\trop(\embalgpoint_1)$.
    The condition on algebraic curves to be tangent to $\Tor_\K(\sigma)$ at $\embalgpoint_2$ is translated to the condition on their tropicalization to have an unbounded end of weight $2$ with direction $(0,-1)$.
    Indeed, the tropicalization of $\langle (\embalgpoint_1, 1), (\embalgpoint_2, 2) \rangle$ as in Definition \ref{def:trop_point_conditions_tropicalization} is $\langle ((0,0), (0,0)), ((0,1)\R, 2\cdot (0,-1)) \rangle$.
\end{example}

\begin{proposition}\label{prop:complex_tangency_tropicalization}
    Fix an integer $g\ge 0$ and a lattice polygon $P$, and denote by $|\partial P|$ the number of lattice points on the boundary of $P$.
    Let $\left\langle (\embalgpoint_i, m_i) \right\rangle_{i=1}^n$ be point conditions in $\Tor_\K(P)$ with $\sum_{i=1}^n m_i = |\partial P| + g - 1$, in general position among configurations of algebraic point conditions with the given type, and assume that its tropicalization $\left\langle (\embtroppoint_i, u_i) \right\rangle_{i=1}^n$ is in general position among tropical point conditions with the same type.
    Let $C$ be a genus $g$ curve of degree $d$ satisfying $\left\langle (\embalgpoint_i, m_i) \right\rangle_{i=1}^n$, and let $\Gamma$ be the tropicalization of $C$.
    Then $\Gamma$ satisfies the tropical point conditions $\left\langle (\embtroppoint_i, u_i) \right\rangle_{i=1}^n$, it has genus $g$, it is trivalent, regular, and all its unfixed ends have multiplicity $1$.
\end{proposition}

\begin{proof}
    The fact that $\Gamma$ satisfies the tropical point conditions $\langle (\embtroppoint_i, u_i) \rangle_{i=1}^n$ and has genus at most $g$ is a standard result about tropicalization of algebraic curves.
    Let $\Delta$ be the degree of $\Gamma$.
    Since $\Gamma$ satisfies the tropical point conditions $\langle (\embtroppoint_i, u_i) \rangle_{i=1}^n$, we have $|\Delta| \le |\partial P| - \sum_{u_i\ne 0} (m_i-1)$.
    By the assumption on the general position of the tropicalization of point conditions, we have 
    \[ \dim Ev(\openmoduli{[\Gamma]}) = n + |\setc{i}{u_i=0}| \]
    meaning that $\dim \openmoduli{[\Gamma]} \ge n + |\setc{i}{u_i=0}|$ and if there is an equality then $Ev|_{\openmoduli{[\Gamma]}}$ is injective.
    By Lemma \ref{lemma:dim_moduli_space}, we have 
    \begin{align*}
    \dim \openmoduli{[\Gamma]} & \le |\Delta|+g-1+|\setc{i}{u_i=0}| \\
    &\le |\partial P| - \sum_{u_i\ne 0} (m_i-1) + g - 1 + |\setc{i}{u_i=0}| \\
    &= |\partial P| - \sum_{i=1}^{n} m_i +n +g -1 + |\setc{i}{u_i=0}| = n + |\setc{i}{u_i=0}|
    \end{align*}
    so in fact there is equality.
    Thus, by Lemma \ref{lemma:dim_moduli_space}, $\Gamma$ is trivalent of genus $g$ and $|\Delta| = |\partial P| - \sum_{u_i\ne 0} (m_i-1)$ so all the unfixed ends have multiplicity $1$.
    Lemma \ref{lemma:Ev_injective} implies that all components of $\ogamma\setminus \bs{\abstroppoint}$ are of genus $0$ and contain at most one unfixed end, an Euler characteristic calculation shows that there are no components with no unfixed ends, meaning that $\Gamma$ is regular.
\end{proof}

We can now state the main result of this section.
\begin{theorem}\label{thm:complex_tangency_num}
    Let $\left\langle (\embalgpoint_i, m_i) \right\rangle_{i=1}^n$ be point conditions in $\Tor_\K(P)$ that satisfy the assumptions of Proposition \ref{prop:complex_tangency_tropicalization}, and let $\Gamma$ be a genus $g$ tropical curve of degree $\Delta$ satisfying its conclusions.
    Then the number of genus $g$ curves in $|\mathcal{L}_P|$ satisfying $\left\langle (\embalgpoint_i, m_i) \right\rangle_{i=1}^n$ and tropicalizing to $\Gamma$ is equal to
    \begin{equation}\label{eq:complex_tangency_num}
        \prod_{V\in \Gamma^0}\mu(V)\cdot \prod_{E\in \Gamma^1_{\infty}}\wt(E)^{-1}.
    \end{equation}
\end{theorem}

\begin{proof}
    The proof is split into several steps.
    In Lemma \ref{lemma:complex_enhancements_count} we count the number of enhancements of $\Gamma$ that satisfy the algebraic point conditions $\left\langle (\embalgpoint_i, m_i) \right\rangle_{i=1}^n$.
    In Lemma \ref{lemma:number_of_modifications} we count the number of modifications of $\Gamma$ that satisfy the algebraic point conditions $\left\langle (\embalgpoint_i, m_i) \right\rangle_{i=1}^n$.
    Finally, in Lemma \ref{lemma:complex_patchworking} we show that the number of modifications is equal to the desired number of algebraic curves.
\end{proof}

\begin{definition}
    We say that an enhanced plane tropical curve $(\Gamma, h, \bs{\abstroppoint}, \{\varphi_V\}_{V\in \Gamma^0})$ satisfies algebraic point conditions $\left\langle (\embalgpoint_i, m_i) \right\rangle_{i=1}^n$ if:
    \begin{enumerate}
        \item The plane tropical curve $(\Gamma, h, \bs{\abstroppoint})$ satisfies tropical point conditions $\left\langle (\trop(\embalgpoint_i)+u_i\R, m_i\cdot u_i) \right\rangle_{i=1}^n$.
        
        \item For every marked bounded edge $E$, dual to $\sigma$, with endpoints $V_1$ and $V_2$ and marking $\abstroppoint_i$, the mutual point of intersection with the toric boundary $\varphi_{V_1}(\absalgpoint_{V_1,E}) = \varphi_{V_2}(\absalgpoint_{V_2,E})\in \Tor_\C(\sigma)$ is the one defined by $\ini(\embalgpoint_i)$.
        
        \item For every fixed end $E$ of $\Gamma$ dual to $\sigma$ marked by $\abstroppoint_i$, and adjacent to the trivalent vertex $V$, the intersection with the toric boundary $\varphi_{V}(\absalgpoint_{V,E})$ is the one defined by $\ini(\embalgpoint_i)$.
    \end{enumerate}
\end{definition}

\begin{lemma}\label{lemma:complex_enhancements_count}
    Let $(\Gamma, h, \bs{\abstroppoint})$ be a trivalent plane tropical curve that satisfies the tropical point conditions $\left\langle (\trop(\embalgpoint_i)+u_i\R, m_i\cdot u_i) \right\rangle_{i=1}^n$.
    Then there exist
    \begin{equation*}
        \prod_{V\in \Gamma^0}\mu(V)\cdot \prod_{E\in \Gamma^1}\wt(E)^{-1} \cdot \prod_{\substack{E\in \Gamma^1 \\ \abstroppoint_i\in E \\ \embalgpoint_i\in (\K^\times)^2}}\wt(E)^{-1}
    \end{equation*}
isomorphism classes of enhancements of $(\Gamma, h, \bs{\abstroppoint})$ that satisfy the algebraic point conditions $\left\langle (\embalgpoint_i, m_i) \right\rangle_{i=1}^n$.
\end{lemma}

\begin{proof}
    If $V\in\Gamma^0$ is a trivalent vertex dual to a triangle $P'\subset P$, then $\PP^1\to C'\in|{\mathcal L}_{P'}|$ is a peripherally unibranch curve that matches points on the toric divisors dual to the incoming edges $E_1,E_2$ (with respect to the regular orientation on $\Gamma$) incident to $V$; these points are determined by the points of $\bp$ and by previously chosen admissible limit curves. By \cite[Lemma 3.5]{Shustin_tropical_enum05}, there are $\mu(V)(\wt(E_1)\wt(E_2))^{-1}$ choices for such a curve.
\end{proof}

To get a unique lift to an algebraic curve we need to modify the tropical limit as described in \cite{CNRBI}. We will repeat the relevant parts of the construction and definition here for completeness.

\begin{Construction}\label{con:modified_tropical_limit}
    Let $[\bn:(\widehat C,\bs{\absalgpoint})\to\Tor_\K(P)]$ be a curve.
    We define the \emph{modified tropical limit of $C$} to be the parameterized tropical limit of $\widehat C$, together with \emph{modification data} associated to its bounded edges as follows:
    \begin{enumerate}
        \item The modification of an unmarked bounded edge $E$ of weight $\ell>1$ is described in \cite[Sections 3.5 and 3.6]{Shustin_tropical_enum05} and \cite[Section 2.2(2)]{GS2021}. It is shown that if we perform a toric change of coordinates that makes $E$ horizontal, then there exists a unique $\eta \in \K$ such that the coefficient of $y^{\ell-1}$ in the polynomial $F'(x,y):=F(x,y+\eta)$ vanishes, and this results in adding a stratum isomorphic to $\Tor_\C(P_{mod})$ to the degeneration of surfaces, where $$P_{mod}:=\conv\{(-1,0),(1,0),(0,\ell)\}.$$
        We record this toric change of coordinates, the shift $\eta$, and the limit curve $\PP^1 \to C_{mod}\subset\Tor_\C(P_{mod})$ as the modification data associated to $E$.

        \item For a marked edge $E$ of $\Gamma$ of weight $\ell>1$ containing a point $\abstroppoint_i$, assume that the $h$-image of $E$ is horizontal (see Figure \ref{fcn2}(a)), write $\embalgpoint_i=(\zeta_i+O(t^M),\eta_i+O(t^M))$, and perform the coordinate change $y=y_1+\eta_i$. Since the limit curves associated with the trivalent vertices $V_1,V_2$ are smooth at their common point (see \cite[Lemma 3.5]{Shustin_tropical_enum05}), the Newton polygons of their defining polynomials after this coordinate change leave the triangle $P_{mod}=\conv\{(-1,0),(1,0),(0,\ell)\}$ empty (Figure \ref{fcn2}(b), cf. \cite[Figure 2.18]{ItenbergMikhalkinShustin2007}). Furthermore, the new coordinates of the point $\embalgpoint_i$ are $(\zeta_i+O(t^M),O(t^M))$; hence, since $M\gg0$, its tropicalization $\trop(\embalgpoint_i)$ lies on the vertical end oriented to $-\infty$. Taking into account that the union of limit curves in the modification must have arithmetic genus zero, we obtain a tropical modification consisting of two trivalent vertices $V_1'$ and $V_2'$ joined by a horizontal edge $E'$, as shown in Figure \ref{fcn2}(c).
        Its dual subdivision is the two triangles
        $$P_{1,mod}=\conv\{(-1,0),(0,0),(0,\ell)\},\quad P_{2,mod}=\conv\{(0,0),(1,0),(0,\ell)\},$$ (Figure \ref{fcn2}(d)), while the limit curves $\PP^1\to C_{1,mod}$, $\PP^1\to C_{2,mod}$ are peripherally unibranch.
        
        The modification data associated to $E$ is the toric change of coordinates that makes $E$ horizontal, the curves $C_{1, mod}, C_{2, mod}$, and the modification data associated to the edge $E'$ in the modified curve.
    \end{enumerate}
\end{Construction}

\begin{figure}
    \setlength{\unitlength}{1mm}
    \begin{picture}(145,35)(0,0)
        \thinlines
        \dashline{2}(70,30)(100,30)
        \put(35,10){\vector(1,0){30}}\put(50,10){\vector(0,1){20}}
        \put(105,10){\vector(1,0){30}}\put(120,10){\vector(0,1){20}}
        
        \thicklines
        {\color{blue}
            \put(40,10){\line(2,3){10}}\put(40,10){\line(1,0){20}}\put(60,10){\line(-2,3){10}}
            \put(110,10){\line(2,3){10}}\put(110,10){\line(1,0){20}}\put(130,10){\line(-2,3){10}}\put(120,10){\line(0,1){15}}
        }
        {\color{red}
            \put(0,15){\line(1,1){5}}\put(0,25){\line(1,-1){5}}\put(5,20){\line(1,0){15}}
            \put(20,20){\line(1,1){5}}\put(20,20){\line(1,-1){5}}
            \put(75,30){\line(1,-1){5}}\put(80,25){\line(1,0){10}}\put(80,25){\line(0,-1){15}}
            \put(90,25){\line(0,-1){15}}\put(90,25){\line(1,1){5}}
        }
        \put(12,0){(a)}\put(46,0){(b)}\put(83,0){(c)}\put(116,0){(d)}
        \put(50,25){$\ell$}\put(120,25){$\ell$}
        \put(36,6){$-1$}\put(57,6){$1$}\put(106,6){$-1$}\put(127,6){$1$}
        
        \put(8,15){$\embtroppoint_i$}\put(73,14){$\embtroppoint_i$}
        \put(4,21.5){$V_1$}\put(16,21.5){$V_2$}\put(72,26){$V_1$}\put(95,26){$V_2$}
        \put(9,19){$\bullet$}\put(79,14){$\bullet$}
        
    \end{picture}
    \caption{The modification of a marked edge in Construction \ref{con:modified_tropical_limit} and Definition \ref{def:modification_data}. }\label{fcn2}
\end{figure}
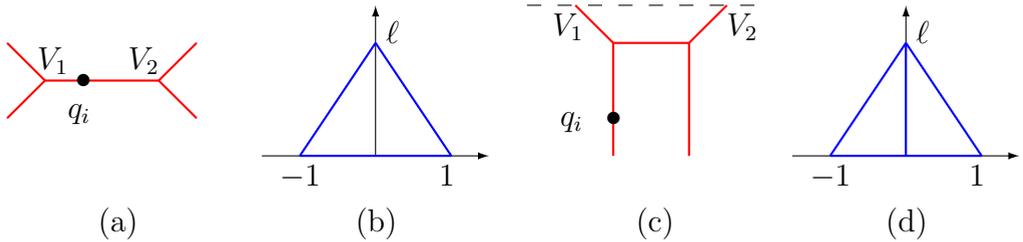

The following definition is the precise description of what limit curves can occur as a modification data of the modified tropical limit.

\begin{definition}\label{def:modification_data}
    Let $(\Gamma, h, \bs{\troppoint}, \{\varphi_V:\PP^1\to C_{V}\subset \Tor_\C(P_V)\}_{V\in \Gamma^0})$ be an enhanced plane tropical curve.
    Let $E\in \Gamma^{(1)}$ be a bounded edge.
    We define \emph{modification data associated to $E$} separately for each case as follows:
    \begin{enumerate}
        \item A modification datum associated to an unmarked edge $E$ of weight $\ell$ is a triple $(\Phi, \zeta, \varphi_{mod})$, where $\Phi$ is a toric change of coordinates that makes $E$ horizontal, $\zeta\in \K$, and $\varphi_{mod}:\PP^1\to C_{mod}\subset \Tor_\C(P_{mod})$ is a rational curve, where $P_{mod}=\conv\{(-1,0),(1,0),(0,\ell)\}$, and the coefficient of $y^{\ell-1}$ in the defining equation of $C_{mod}$ vanishes.
        
        Denote by $\sigma = P_{V_1}\cap P_{V_2}$ the edge of the dual subdivision corresponding to $E$.
        Both $C_{V_1}$ and $C_{V_2}$ intersect $\Tor_\C(\sigma)$ at a unique point $z$.
        We say the modification data is \emph{coherent} if the second coordinate of $z$ (after applying $\Phi$) is $\eta(0)$ and the defining equation of $C_{mod}$ restricted to $\conv\{(-1,0), (0,\ell)\}$ (respectively, $\conv\{(1,0), (0, \ell)\}$) is the local equation of the defining equation of $C_{V_1}$ (respectively, $C_{V_2}$) at the point of intersection with the toric divisor $\Tor_\C(\sigma)$ where $\sigma = P_{V_1}\cap P_{V_2}$ is the edge of the dual subdivision corresponding to $E$.
        Explicitly, in coordinates on $\Tor_\C(P_{V_i})$ where $\Tor_\C(\sigma)$ is defined by $x=0$, the limit curve $C_{V_i}$ ($i=1,2$) intersects $\Tor_\C(\sigma)$ at a single point $z=(0,\eta(0))$, where its defining equation is of the form $(y-\eta(0))^\ell + \theta_i(y) x + O(x^2)$ for some $\theta_i\in \C[y]$. Thus coherence means that the restriction of the defining equation of $C_{mod}$ to $\conv\{(-1, 0), (0, \ell)\}$ (respectively $\conv\{(1,0), (0, \ell)\}$) is $y^\ell + \theta_i(\eta(0)) x^{-1}$ (respectively $y^\ell + \theta_i(\eta(0))x$).

        \item
        A modification datum associated to an edge $E$ of $\Gamma$ of weight $\ell>1$ containing a point $\abstroppoint_i$ is a tuple $(\Phi, \varphi_{1,mod}, \varphi_{2,mod}, \Omega)$, where $\Phi$ is a toric change of coordinates that makes $E$ horizontal, $\varphi_{1,mod}:\PP^1\to C_{1,mod}\subset \Tor_\C(P_{1,mod})$ and $\varphi_{2,mod}:\PP^1\to C_{2,mod}\subset \Tor_\C(P_{2,mod})$ are rational peripherally unibranch curves, where $P_{1,mod}=\conv\{(-1,0),(0,\ell),(0,0)\}$ and $P_{2,mod}=\conv\{(0,0),(1,0),(0,\ell)\}$, and $\Omega$ is modification data associated to the unmarked edge dual to $\conv\{(0,0), (0,\ell)\}$.

        This data is called \emph{coherent} if $\Omega$ is coherent, the intersection of $C_{1,mod}$ with the toric divisor $\Tor_\C(\conv\{(-1,0), (0,0)\})$ is $\xi_i$, the initial term of the first coordinate of the marked point $p_i$ after the toric coordinate change, and additionally the defining equation of $C_{1,mod}$ (respectively $C_{2,mod}$) restricted to $\conv\{(-1,0), (0,\ell)\}$ (respectively $\conv\{(0,\ell), (1, 0)\}$) is the local equation of the defining equation of $C_{V_1}$ (respectively $C_{V_2}$) at the point of intersection with the toric divisor $\Tor_\C(\sigma)$, where $\sigma = P_{V_1}\cap P_{V_2}$ is the edge of the dual subdivision corresponding to $E$.
    \end{enumerate}
\end{definition}

\begin{remark}
    It is obvious that the modification data constructed in Construction \ref{con:modified_tropical_limit} are coherent.
\end{remark}

\begin{definition}
    We say that a pair of modification data associated with a bounded edge $E$ is \emph{equivalent} if it arises by applying Construction \ref{con:modified_tropical_limit} to the same algebraic curve with different choices of toric change of coordinates.
\end{definition}

\begin{lemma}\label{lemma:number_of_modifications}
    Let $(\Gamma, h, \bs{\abstroppoint}, \{C_V\}_{V\in \Gamma^0})$ be an enhanced plane tropical curve satisfying the point constraints.
    Then there exist
    $$\prod_{E\in \Gamma_b^1}\wt(E) \cdot \prod_{\substack{i \\ \abstroppoint_i \in E \in \Gamma_b^1}}\wt(E)$$
    equivalence classes of coherent modification data associated to $\Gamma$.
\end{lemma}

\begin{proof}
    We enumerate the possible coherent modification data associated with each bounded edge, as in Definition \ref{def:modification_data}.

    \begin{enumerate}

        \item For an unmarked edge $E$ of $\Gamma$ of weight $\ell$, the admissible modified limit curve is $\PP^1\to C'\subset\Tor(P')$, where $$P'=\conv\{(-1,0),(1,0),(0,\ell)\}$$ (see Figure \ref{fcn2}(b)). The constraints are fixed points on the toric divisors $\Tor([(-1,0),(0,\ell)])$ and $\Tor([(1,0),(0,\ell)])$, determined by the chosen admissible limit curves of the embedded part, and the condition that the coefficient of $y^{\ell-1}$ vanishes (cf. \cite[Lemma 3.9]{Shustin_tropical_enum05} and \cite[Lemma 2.49]{ItenbergMikhalkinShustin2007}). These constraints are satisfied by exactly $\ell$ curves \cite[Lemma 3.9]{Shustin_tropical_enum05}.
    
        \item For an edge of $\Gamma$ of weight $\ell>1$ containing a point $\abstroppoint_i$, we should construct two admissible limit curves $C_{1,mod},C_{2,mod}$, as well as modification data for an unmarked edge of weight $\ell$, as described in item 2 of Definition \ref{def:modification_data}.
        A curve $C_{1,mod}\in|{\mathcal L}_{P_{1,mod}}|$ must match a point on the divisor $\Tor([(-1,0),(0,\ell)])$, prescribed by previously chosen admissible embedded limit curves, and the point $\xi_i\in\Tor([(-1,0),(0,0)])$, and it should intersect the divisor $\Tor([(0,0),(0,\ell)])$ at one point. Its equation is $(y-\tau)^\ell+ax^{-1}=0$ with fixed $a$ and fixed $\tau^\ell$, which yields $\ell$ possibilities for $C_{1,mod}$. In turn, the curve $C_{2,mod}\in|{\mathcal L}_{P_{2,mod}}|$ is uniquely determined by the prescribed points on the divisors $\Tor([(0,0),(0,\ell)])$ and $\Tor([(1,0),(0,\ell)])$.
        At last, the modification data of the unmarked edge have $\ell$ options as shown above.
    \end{enumerate}
\end{proof}

Next we describe how one can glue the modification data to get a curve satisfying the point conditions and tropicalizing to a given tropical curve.
This has been described in various sources. We follow the proof of \cite[Lemma 3.13]{CNRBI}, while providing a more detailed explanation.

\begin{Construction}\label{con:surface_family}
    Let $U\subset (\C((t))^\times)^n$ be a plane, i.e., a two-dimensional subvariety defined by linear equations, and suppose the projection $U\to (\C((t))^\times)^2$ is injective.
    Let $h:\Gamma\to \R^n$ be a tropical curve whose image is contained in $\trop(U)$ and whose vertices all have integer coordinates.
    We will define a family of surfaces $\mathcal{S}\to \mbb{A}^1$ as follows:
    \begin{enumerate}
        \item Consider a rational polyhedral subdivision $\Sigma$ of $\trop(U)$ that contains all vertices and edges of $\Gamma$ as cells.
        \item Consider the family of $n$-folds $\mathcal{X}\to \mbb{A}^1$ defined as in \cite{Nishinou2004ToricDO}, i.e. the toric variety corresponding to the cone over $\Sigma$.
        Note that since $\Sigma$ is not a subdivision of $\R^n$, the toric variety $\mathcal{X}$ is not proper over $\mbb{A}^1$.
        \item Consider the map 
        \[ \Spec(\C((t))[x_1^{\pm 1}, \ldots, x_n^{\pm 1}]) \to \Spec(\C[t^{\pm 1}, x_1^{\pm 1}, \ldots, x_n^{\pm 1}]) = (\C^\times)^{n+1} \subset \mathcal{X} \]
        induced by the inclusion $\C[t^{\pm 1}] \subset \C((t))$.        
        We define $\mathcal{S}$ as the closure of the image of $U$ in $\mathcal{X}$ under this map.
    \end{enumerate}

    In addition we define the line bundle $\mathcal{L}$ on $\mathcal{S}$ as the pullback of the line bundle $\mathcal{O}_{\mathcal{X}}(1)$ on $\mathcal{X}$ under the inclusion.
\end{Construction}

The following lemma follows immediately from Tevelev's tropical compactification theorem \cite[Proposition 2.3]{tevelev2007compactifications}, and the fact that $\mathcal{X}$ is a toric degeneration.
\begin{lemma}\label{lemma:surface_family_proper}
    The family $\mathcal{S}\to \mbb{A}^1$ as in Construction \ref{con:surface_family} is a proper family of surfaces satisfying conditions X1, X2 of \cite{shustin2006patchworking}.
\end{lemma}

We are now ready to state and prove the patchworking lemma in our context.
\begin{lemma}\label{lemma:complex_patchworking}
    Let $(\Gamma, h, \bs{\abstroppoint}, \{\varphi_V\}_{V\in \Gamma^0})$ be an enhanced tropical curve satisfying the point conditions.
    For any coherent modification data associated to $\Gamma$ there exists a unique curve $[\bn:(\widehat C,\bs{\absalgpoint})\to\Tor_\K(P)]$ satisfying the point conditions and whose modified tropical limit (in the sense of Construction \ref{con:modified_tropical_limit}) has this modification data.
\end{lemma}

\begin{proof}
    The proof follows the same lines as the proof of \cite[Lemma 3.13]{CNRBI}.
    The modification data defines a linear map $(\C((t))^\times)^2 \to \C((t))^n$, obtained by composing the toric change of coordinates with the coordinate changes $y\mapsto y+\eta$ and $y\mapsto y+\eta_i$ described in Construction \ref{con:modified_tropical_limit}.
    This defines a plane $U\subset (\C((t))^\times)^n$ as in Construction \ref{con:surface_family}.
    In fact, the modification data uniquely determines the tropicalization $\Gamma\subset \trop(U)$ of the image of possible curves $\widehat C$ in $U$; see Figure \ref{fcn2} and \cite[Section 2.5.8]{ItenbergMikhalkinShustin2007}.
    By performing a base change over $\Spec(\C((t^{1/m}))) \to \Spec(\C((t)))$ we can assume that all the vertices of $\Gamma$ are integral.
    Applying Construction \ref{con:surface_family} to $\Gamma$ and $U$, we obtain, by Lemma \ref{lemma:surface_family_proper}, a proper family of surfaces $\mathcal{S}\to \mbb{A}^1$ satisfying conditions X1 and X2 of \cite{shustin2006patchworking}.

    The limit curves $\varphi_{\text{mod}}, \varphi_{1, \text{mod}} $, and $\varphi_{2, \text{mod}}$ of the modification data define a section $\xi_0(\mathcal{S}_0, \mathcal{L}_0)$. 
    By construction, this section satisfies conditions S1, S2, S3, S4, S5 of \cite{shustin2006patchworking}.
    Finally, it was shown in the proof of \cite[Lemma 3.13]{CNRBI} that the section $\xi_0(\mathcal{S}_0, \mathcal{L}_0)$ satisfies condition T1, T2, T3 of \cite[Theorem 2.15]{shustin2006patchworking}, so by the result of said theorem we get the existence of the curve $[\bn:(\widehat C,\bs{\absalgpoint})\to \mathcal{S}]$.
    The projection $U\to (\C((t))^\times)^2$ is birational, so we get the desired curve $[\bn:(\widehat C,\bs{\absalgpoint})\to\Tor_\K(P)]$.
\end{proof}

\begin{example}
    One can follow the computations of this section to get a signed count (with Welschinger sign) of real curves satisfying the point conditions and tropicalizing to a given tropical curve.
    Indeed, using the results of \cite{Shustin_tropical_enum05} (see also Lemmas \ref{lemma:real_enhancements_count} and \ref{lemma:odd_in_real_modification_count} below), one can show that if the only even edges of $\Gamma$ are the fixed ends, then the corresponding signed count is
    \[ \prod_{V\in \Gamma^0} (-1)^{\Int(V)} \]
    where $\Int(V)$ is the number of integer points in the interior of the dual cell to $V$.

    Unlike the count of complex curves, this signed count is not invariant in general.
    Indeed, consider the polytope $P = \conv\{(0,2), (2,0), (5,1), (2,3)\}$ and the tropical point conditions
    $$\left\langle ((-1,-1)\R, (-2,-2)),((1,-3), (0,0)), ((1,y), (0,0)) \right\rangle.$$
    For $y>1.5$ we will get a unique tropical curve, depicted in Figure \ref{fig:non_invariance_real_count}(a) with a positive Welschinger sign.

    For $y<1.5$ we will get the tropical curves depicted in Figure \ref{fig:non_invariance_real_count}(b) and \ref{fig:non_invariance_real_count}(c) with opposite Welschinger signs.

\end{example}

\begin{figure}
    
	\begin{center}
        \begin{tikzpicture}
            \pgfmathsetmacro{\marksize}{0.05}
            \tikzset{
                every path/.style = {line width=0.7pt, red}, 
                every node/.style = {black},
                pics/x mark/.style = {
                    code={
                        \draw [black] (-\marksize, -\marksize) -- (\marksize, \marksize);
                        \draw [black] (-\marksize, \marksize) -- (\marksize, -\marksize);
                    }}
            }
            \begin{scope}[scale=0.4, decoration={markings,mark=at position 0.5 with {\arrow{>}}}, xshift=-10, yshift=-20]
                \draw [postaction={decorate}] (-4, -4) --  (0,0);
                \draw [black] (-4+2*\marksize, -4-2*\marksize) -- (-4-2*\marksize, -4+2*\marksize);
                \path (1, -3) pic[rotate=20] {x mark};
                \draw [postaction={decorate}] (1, -3) -- (0,0);
                \draw [postaction={decorate}] (1,-3) -- (1.33, -4);
                \draw [postaction={decorate}] (0,0) -- (0.5, 2.5);
                \draw [postaction={decorate}] (0.5, 2.5) -- (-0.625, 4.75);
                \draw [postaction={decorate}] (1.5, 4) -- (0.5, 2.5);
                \path (1.5, 4) pic[rotate=45] {x mark};
                \draw [postaction={decorate}] (1.5, 4) -- (2, 4.75);

                \node at (-1, -4) {$(a)$};
            \end{scope}

            \begin{scope}[scale=0.6, every path/.style={blue, thick}, yshift=-200, xshift=-100];
                \draw (0,2) -- (2,0) -- (5,1) -- (2,3) -- cycle;
                \draw (0,2) -- (5,1);
                \draw [black, fill=black] (2,0) circle (0.03);
                \draw [black, fill=black] (1,1) circle (0.03);
                \draw [black, fill=black] (2,1) circle (0.03);
                \draw [black, fill=black] (3,1) circle (0.03);
                \draw [black, fill=black] (4,1) circle (0.03);
                \draw [black, fill=black] (5,1) circle (0.03);
                \draw [black, fill=black] (0,2) circle (0.03);
                \draw [black, fill=black] (1,2) circle (0.03);
                \draw [black, fill=black] (2,2) circle (0.03);
                \draw [black, fill=black] (3,2) circle (0.03);
                \draw [black, fill=black] (2,3) circle (0.03);
            \end{scope}

            \begin{scope}[scale=0.5, decoration={markings,mark=at position 0.5 with {\arrow{>}}}, xshift=250]
                \draw [postaction={decorate}] (-4, -4) --  (-1/6, -1/6);
                \draw [black] (-4+2*\marksize, -4-2*\marksize) -- (-4-2*\marksize, -4+2*\marksize);
                
                \path (1, 1.25) pic[rotate=45] {x mark};
                \draw [postaction={decorate}] (1, 1.25) -- (2, 2.75);
                \draw [postaction={decorate}] (1,1.25) -- (1/18, -1/6);

                \path (1, -3) pic[rotate=20] {x mark};
                \draw [postaction={decorate}] (1,-3) -- (1.33, -4);
                \draw [postaction={decorate}] (1,-3) -- (1/18, -1/6);

                \draw (1/18, -1/6) -- (-1/6, -1/6);
                \draw [postaction={decorate}] (-1/6, -1/6) -- (-39/24, 2.75);

                \node at (-1, -4) {$(b)$};
            \end{scope}

            \begin{scope}[scale=0.6, every path/.style={blue, thick}, yshift=-200, xshift=100]
                \draw (0,2) -- (2,0) -- (5,1) -- (2,3) -- cycle;
                \draw (2,0) -- (2,3);
                \draw [black, fill=black] (2,0) circle (0.03);
                \draw [black, fill=black] (1,1) circle (0.03);
                \draw [black, fill=black] (2,1) circle (0.03);
                \draw [black, fill=black] (3,1) circle (0.03);
                \draw [black, fill=black] (4,1) circle (0.03);
                \draw [black, fill=black] (5,1) circle (0.03);
                \draw [black, fill=black] (0,2) circle (0.03);
                \draw [black, fill=black] (1,2) circle (0.03);
                \draw [black, fill=black] (2,2) circle (0.03);
                \draw [black, fill=black] (3,2) circle (0.03);
                \draw [black, fill=black] (2,3) circle (0.03);
            \end{scope}

            \begin{scope}[scale=0.5, decoration={markings,mark=at position 0.5 with {\arrow{>}}}, xshift=490]
                \draw [postaction={decorate}] (-4, -4) --  (0.5,0.5);
                \draw [black] (-4+2*\marksize, -4-2*\marksize) -- (-4-2*\marksize, -4+2*\marksize);
                \draw [postaction={decorate}] (1, 1.25) -- (0.5,0.5);
                \path (1, 1.25) pic[rotate=45] {x mark};
                \draw [postaction={decorate}] (1, 1.25) -- (2, 2.75);
                \draw [postaction={decorate}] (0.5,0.5) -- (0.5, -1.5);
                \path (1, -3) pic[rotate=20] {x mark};
                \draw [postaction={decorate}] (1, -3) -- (0.5,-1.5);
                \draw [postaction={decorate}] (1,-3) -- (1.33, -4);
                \draw [postaction={decorate}] (0.5, -1.5) -- (-1.625, 2.75);

                \node at (-1, -4) {$(c)$};
            \end{scope}

            \begin{scope}[scale=0.6, every path/.style={blue, thick}, yshift=-200, xshift=300]
                \draw (0,2) -- (2,0) -- (5,1) -- (2,3) -- cycle;
                \draw (2,0) -- (4,1) -- (2,3);
                \draw (4,1) -- (5,1);
                \draw [black, fill=black] (2,0) circle (0.03);
                \draw [black, fill=black] (1,1) circle (0.03);
                \draw [black, fill=black] (2,1) circle (0.03);
                \draw [black, fill=black] (3,1) circle (0.03);
                \draw [black, fill=black] (4,1) circle (0.03);
                \draw [black, fill=black] (5,1) circle (0.03);
                \draw [black, fill=black] (0,2) circle (0.03);
                \draw [black, fill=black] (1,2) circle (0.03);
                \draw [black, fill=black] (2,2) circle (0.03);
                \draw [black, fill=black] (3,2) circle (0.03);
                \draw [black, fill=black] (2,3) circle (0.03);
            \end{scope}
        \end{tikzpicture}
    \end{center}
    \caption{An example of non-invariance of the signed count of real curves with tangency conditions. In the top row we depict the tropical curve and in the bottom row we depict the dual subdivisions. The black dots in the dual subdivision picture denote the integer points.}\label{fig:non_invariance_real_count}
\end{figure}
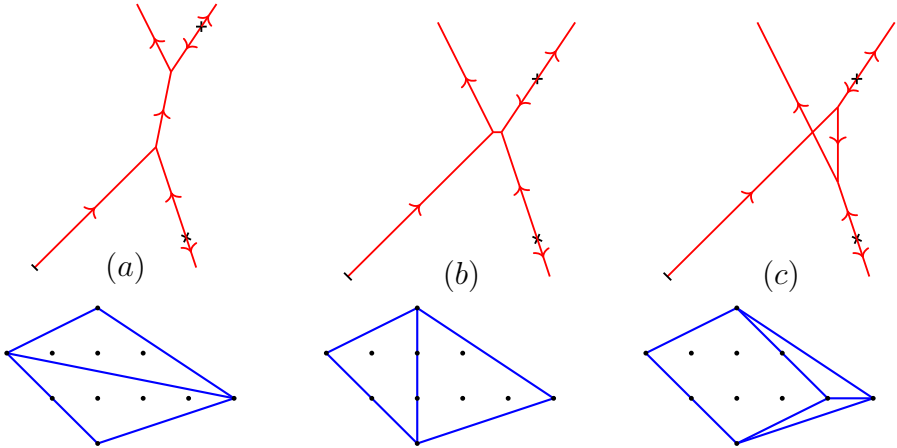

\section{Real correspondence theorem}\label{sec:real_conjugate}

We now focus on the real setting and provide a tropical signed count of real curves satisfying a given set of essentially peripheral algebraic point conditions.
\begin{definition}
    Essentially peripheral real point conditions are a sequence of points $\embalgpoint_i\in \left(\K_\R^\times\right)^2 \cup \partial \Tor_\K(P)$ such that the set $\{\embalgpoint_i\}$ is invariant under conjugation.

    The reduced tropicalization of essentially peripheral real point conditions is obtained by taking the tropicalization of these point conditions as in Definition \ref{def:trop_point_conditions_tropicalization}, and then, for every pair of complex conjugate points $\embalgpoint_i$ and $\embalgpoint_j$, replacing the equal pairs $(\trop(\embalgpoint_i)+u_i\R, u_i)$ and $(\trop(\embalgpoint_j)+u_j\R, u_j)$ in the resulting tropical point conditions by a single pair $(\trop(\embalgpoint_i)+u_i\R, 2u_i)$.
\end{definition}

\begin{lemma}\label{lemma:real_tropicalization}
    Let $\bs{\embalgpoint}$ be essentially peripheral real point conditions consisting of $n=|\partial P| + g - 1$ points.
    Suppose that $\bs{\embalgpoint}$ is in general position among the configurations of essentially peripheral real point conditions with the given type, and moreover its reduced tropicalization is in general position among the configurations of tropical point conditions with the given type.

    Let $C$ be a real genus $g$ curve of degree $d$ passing through $\bs{\embalgpoint}$, let $\Gamma$ be the tropicalization of $C$, and let $\Gamma':=\Gamma/\bc$ be its natural quotient.
    Then the following holds for $\Gamma$ and $\Gamma'$:
    \begin{enumerate}
        \item $\Gamma'$ satisfies the reduced tropicalization of $\bs{\embalgpoint}$.
        \item $\Gamma'$ is trivalent, regular, and has genus $g$.
        \item $\Gamma$ is regular of genus $g$ and all its ends have weight $1$.
    \end{enumerate}
\end{lemma}

\begin{proof}
    The first item, and the fact that the genus of $\Gamma$ is at most $g$, follow immediately from standard results on tropicalization of algebraic curves.
    Denote the number of conjugate pairs of points in $\bs{\embalgpoint}$ by $s$.
    By the assumption on the general position of the reduced tropicalization of point conditions, we have 
    \[ \dim Ev(\openmoduli{[\Gamma']}) = 2(n-2s)+s=2n-3s=2|\partial P| +2g-2-3s.\]

    Denote by $\Delta'$ the degree of $\Gamma'$.
    Then $|\Delta'| \le |\partial P| - s$, so by Lemma \ref{lemma:dim_moduli_space}, we have 
    \begin{align*}
    \dim \openmoduli{[\Gamma']} & \le |\Delta'|+g-1+(n-2s) \le \\
    & \le |\partial P| - s + g - 1 + n - 2s = 2|\partial P| + 2g - 2 - 3s
    \end{align*}
    meaning that there is an equality.
    Thus, by Lemma \ref{lemma:dim_moduli_space}, $\Gamma'$ is trivalent of genus $g$ and $|\Delta'| = |\partial P| - s$ so all the ends of $\Gamma$ have weight $1$.
    Additionally, Lemma \ref{lemma:Ev_injective} implies that $\Gamma'$ is regular, similarly to the proof of Proposition \ref{prop:complex_tangency_tropicalization}.
    This concludes the proof of item $(2)$.

    The fact that $\Gamma$ is regular of genus $g$ follows immediately from the properties of $\Gamma'$.
\end{proof}

\begin{lemma}\label{lemma:4_valent_real_enhancement}
    Let $\delta\subset \R^2$ be a nondegenerate lattice triangle with an edge $\sigma_0$ of even length and two edges $\sigma_1, \sigma_2$ of odd length.
    Given a real point in $\Tor_\R(\sigma_1)$, and two complex conjugate points on $\Tor_\C(\sigma_0)$, there exist exactly $\frac{\mu(V)}{\|\sigma_0\|}$ real rational curves $C_k\in|{\mathcal L}_{\delta}|$ that pass through the given points and one more point on $\Tor_\R(\sigma_2)$ and are unibranch at all these four points.
    The number of elliptic nodes of each of these curves is $|\Int(\delta)\cap\Z^2|\mod 2$.
\end{lemma}

\begin{proof}
    
By an automorphism of $\Z^2$ we can bring $\delta$ to the position
$$\conv\{(0,m_0),\ (m_1,0),\ (m_1+2m_2,0)\},\quad\sigma_1=[(0,m_0),(m_1,0)],$$ $$m_0,m_1,m_2>0,\quad m_0<m_1+2m_2,\quad\gcd(m_0,m_1)\equiv1\mod2.$$
Thus, we can choose a real parametrization of $C\in|{\mathcal L}_{\delta}|$ in the form
$$x=a\rho^{m_0},\quad y=b\rho^{m_1}(\rho-\tau)^{m_2}(\rho-\overline\tau)^{m_2},\quad a,b\in\R^*,\ \tau\in\C\setminus\R,$$
with the parameter $\rho\in\C$. The only possible real reparametrizations are $\rho\mapsto\lambda\rho$, $\lambda\in\R^*$. Assuming that $|\tau|=1$, we leave the only reparametrization $\rho\mapsto-\rho$. Denote $k=\gcd(m_0,m_1)$, $m'_0=m_0/k$, $m'_1=m_1/k$.
The conditions for passing through given points on $\Tor_\C([(0,m_0),(m_1,0)])$ and $\Tor_\C([(m_1,0),(m_1+2m_2,0)])$ read
$$b^{m'_0}+\gamma a^{m'_1}=0,\quad a\tau^{m_0}=\xi$$
with fixed generic $\gamma\in\R^*$ and $\xi\in\C\setminus\R$.

Suppose that $m_0$ is odd. Then (up to a reparametrization $\rho\mapsto-\rho$) we obtain $m_0$ solutions
$$a=|\xi|,\quad \tau=\left(\frac{\xi}{|\xi|}\right)^{1/m_0},\quad b=\left(-\gamma|\xi|^{m'_1}\right)^{1/m'_0}\in\R^*.$$
Suppose that $m_0$ is even. Then $m'_0$ is even, and $m'_1$ is odd. Then we obtain $m_0$ solutions
$$a=\pm|\xi|,\ a\gamma<0,\quad \tau=\left(\frac{\xi}{a}\right)^{1/m_0},\quad b=\pm\left(-\gamma a^{m'_1}\right)^{1/m'_0}\in\R^*,$$
where the opposite roots in the formula for $\tau$ correspond to the same curve in view of the reparametrization $\rho\mapsto-\rho$.

It is easy to see that each curve $C$ is smooth at the intersection points with the toric divisors, and its real part is diffeomorphic to $S^1$. Hence, its nodes are elliptic or complex conjugate as required.
\end{proof}

\begin{definition}\label{def:real_enhancement_sign}
    Let $\Gamma$ be a real plane tropical curve and let $\Gamma'$ be its natural quotient by the action of $\bc$.
    Consider a conjugation invariant enhanced tropical curve $(\Gamma, h, \bs{\abstroppoint}, \{\varphi_V\}_{V\in \Gamma^0})$ that tropicalizes to $\Gamma$.
    If $\delta$ is a two dimensional cell of the subdivision dual to $\Gamma'$, then the enhancement induces a plane curve $C_{\delta}\in|{\mathcal L}_{\delta}|$ by taking the product of the equations obtained from the enhancement on the preimages in $\Gamma$ of the vertex of $\Gamma'$ corresponding to $\delta$.
    We define the Welschinger sign of the enhancement as $(-1)^s$, where $s$ is the total number of elliptic nodes in all the curves $C_{\delta}$, where $\delta$ runs over all two dimensional cells of the subdivision dual to $\Gamma'$.
\end{definition}

\begin{lemma}\label{lemma:real_enhancements_count}
    Let $\bs{\embalgpoint}$ be essentially peripheral real point conditions satisfying the assumptions of Lemma \ref{lemma:real_tropicalization}.
    Let $\Gamma$ be a plane tropical curve that satisfies the tropicalization of $\bs{\embalgpoint}$, so that the natural quotient $\Gamma':=\Gamma/\bc$ satisfies the reduced tropicalization of $\bs{\embalgpoint}$, and assume that $\Gamma$ and $\Gamma'$ satisfy the conclusions of Lemma \ref{lemma:real_tropicalization}.
    Assume additionally that $\Gamma$ has no edges of even weight.
    Denote by $\Gamma'_{\text{re}}$ the image in $\Gamma'$ of the subgraph of $\Gamma$ fixed by $\bc$ and by $\Gamma'_{\text{im}}$ the closure of its complement in $\Gamma'$.
    Then there exist
    \begin{equation*}
        \prod_{V\in \Gamma'^{0}_{\text{im}}} \mu(V)\prod_{E\in \Gamma'^{1}_{\text{im}}} \wt(E)^{-1}
    \end{equation*}
    isomorphism classes of conjugation invariant enhanced tropical curves that satisfy $\bs{\embalgpoint}$ and tropicalize to $\Gamma$.
    All of those enhanced tropical curves have the same Welschinger sign
    \begin{equation*}
        \prod_{V\in \Gamma'^{0}_{\text{im}}\setminus\Gamma'^{0}_{\text{re}}} (-1)^{\frac{\mu(V)}{4}}\cdot\prod_{V\in \Gamma'^{0}_{\text{re}}} (-1)^{\Int(V)}
    \end{equation*}
    where $\Int(V)$ is the number of integer points in the interior of the triangle dual to $V$.
\end{lemma}

\begin{proof}
    For a vertex $V\in \Gamma'^{0}$, we will count the number of limit curves associated to preimages of $V$ in $\Gamma$ and the contribution of this vertex to the Welschinger sign.
    We will separate into cases depending on whether $V \in \Gamma'^{0}_{\text{im}}$ and $V \in \Gamma'^{0}_{\text{re}}$.
    \begin{description}
        \item[$V \in \Gamma'^{0}_{\text{re}}\setminus\Gamma'^{0}_{\text{im}}$] 
        This case is identical to the complex case in Lemma \ref{lemma:complex_enhancements_count}. 
        The vertex $V$ has a unique preimage in $\Gamma$, denote it by $V$ as well. 
        Following the proof of \cite[Lemma 3.5]{Shustin_tropical_enum05} (see also \cite[Proposition 6.1]{Shustin_tropical_enum05}), one sees that if the boundary conditions defined by the incoming incident edges of $V$ (with respect to the regular orientation) are real, then among the limit curves enumerated by Lemma \ref{lemma:complex_enhancements_count} there is exactly one real limit curve with no hyperbolic nodes.

        \item[$V \in \Gamma'^{0}_{\text{im}}\cap \Gamma'^{0}_{\text{re}}$] 
        The vertex $V$ has a unique preimage in $\Gamma$, denote it by $V$ as well. 
        The vertex $V$ cannot be a $5$-valent vertex of $\Gamma$, since then it would have an adjacent edge of even weight by the balancing condition; hence $V$ is a $4$-valent vertex of $\Gamma$.
        Since both $\Gamma$ and $\Gamma'$ are regular by Lemma \ref{lemma:real_tropicalization}, the two edges of $\Gamma$ adjacent to $V$ that are not fixed by $\bc$ are oriented towards $V$ via the regular orientation.
        Denote by $\sigma_0$ the edge of the dual subdivision corresponding to the unique edge of $\Gamma'_{\text{im}}$ incident to $V$ and by $\sigma_1, \sigma_2$ the other two edges of the triangle dual to $V$.
        Thus we need to enumerate the number of real rational curves $C\in|{\mathcal L}_{\delta_V}|$ that pass through a given real point in $\Tor_\R(\sigma_1)$, and two complex conjugate points on $\Tor_\C(\sigma_0)$, and one more point on $\Tor_\R(\sigma_2)$ and are unibranch at all these four points.
        By Lemma \ref{lemma:4_valent_real_enhancement}, there are exactly $\frac{\mu(V)}{\|\sigma_0\|}$ such curves and the parity of the number of elliptic nodes of each of these curves is $|\Int(\delta)\cap\Z^2|\mod 2$.

        \item[$V \in \Gamma'^{0}_{\text{im}}\setminus\Gamma'^{0}_{\text{re}}$] 
        The vertex $V$ has $2$ preimages in $\Gamma$, denote them by $V_1$ and $V_2$.
        Since the enhanced curve is required to be compatible with conjugation, the limit curve corresponding to $V_2$ is determined by the limit curve corresponding to $V_1$, for which there are 
        \[ \frac{\mu(V_1)}{\wt(E_1)\wt(E_2)} = \frac{\mu(V)/4}{\wt(E_1')/4\cdot \wt(E_2')/2} = \frac{\mu(V)}{\wt(E_1')\wt(E_2')} \]
        choices by Lemma \ref{lemma:complex_enhancements_count}, where $E_1, E_2$ are the edges of $\Gamma$ adjacent to $V_1$ and oriented towards $V_1$ via the regular orientation, and $E_1', E_2'$ are their images in $\Gamma'$.
        
        The elliptic nodes of the curve $C_{V}$ are exactly the real intersection points of the images of $\varphi_{V_1}$ and $\varphi_{V_2}$ in $\Tor_\C(\delta_V)$, and their number has the same parity as the number of such intersection points (since all other intersection points come in complex-conjugate pairs).
        The number of those intersection points is equal to $\mu(V_1)^2$ by Bernstein-Kushnirenko-Khovanskii theorem.
        Thus the contribution of this vertex to the Welschinger sign is 
        \[ (-1)^{\mu(V_1)^2} = (-1)^{\mu(V_1)} = (-1)^{\mu(V)/4}. \]
    \end{description}

    The proof is concluded by noting that we have divided by the weight of every even edge once (when treating the vertex for which this edge is incoming), and that the curve $C_{\delta}$ for a parallelogram $\delta$ has no elliptic nodes, so it does not contribute to the Welschinger sign.
\end{proof}

\begin{definition}
    Let $(\Gamma, h, \bs{\abstroppoint}, \{\varphi_V\}_{V\in \Gamma^0})$ be a conjugation invariant enhanced tropical curve.
    Modification data for $(\Gamma, h, \bs{\abstroppoint}, \{\varphi_V\}_{V\in \Gamma^0})$ as in Definition \ref{def:modification_data} is called conjugation invariant if for every edge $E$ of $\Gamma$ the modification data associated to $\bc(E)$ is conjugate to the modification data associated to $E$.
    In particular, the modification data associated to an edge fixed by $\bc$ is real.

    The Welschinger sign of a modified tropical curve is defined as the Welschinger sign of the enhancement times $(-1)^r$ where $r$ is the total number of elliptic nodes in all the modified curves.
\end{definition}

\begin{lemma}\label{lemma:even_in_real_modification_vanish}
    Let $(\Gamma, h, \bs{\abstroppoint}, \{\varphi_V\}_{V\in \Gamma^0})$ be a conjugation invariant enhanced tropical curve.
    Let $E$ be a bounded edge of $\Gamma$ of even weight.
    Then the sum of Welschinger signs of all conjugation invariant modification data for $(\Gamma, h, \bs{\abstroppoint}, \{\varphi_V\}_{V\in \Gamma^0})$ that have modification data associated to $E$ vanishes.
\end{lemma}

\begin{proof}
    Since the modification data associated to a marked edge contains modification data associated to an unmarked edge of the same weight, it is enough to consider the case of an unmarked edge.
    Denote the weight of the edge by $2\ell$.

    Let $P_{mod}=\conv\{(-1,0),(1,0),(0,2\ell)\}$.
    We are looking for rational curves $C_{mod}$ whose defining equation is of the form $b_{-1}x^{-1} + g(y) +b_1x$ for fixed $b_{-1}, b_1\in \R^{\times}$ and $g(y)\in \R[y]$ is a monic polynomial with vanishing coefficient of $y^{2\ell-1}$.
    Following the proof of \cite[Lemma 3.9]{Shustin_tropical_enum05}, we can see that the solutions for $g(y)$ are $2\sqrt{b_1b_{-1}}(-1)^k T_{2\ell}\left(\zeta_{4\ell}^k \frac{y}{2\sqrt[2\ell]{2\sqrt{b_1b_{-1}}}}\right)$, where $T_{2\ell}$ is the $2\ell$-th Chebyshev polynomial of the first kind, $\zeta_{4\ell}$ is a primitive $4\ell$-th root of unity, and $k=0,1,\ldots, 2\ell-1$.
    If $b_1b_{-1}<0$, there are no real solutions, and if $b_1b_{-1}>0$, there are exactly two real solutions corresponding to $k=0$ and $k=\ell$ (recall that $T_{2\ell}$ is even, so $T_{2\ell}(ix)$ is real).
    The choice $k=0$ gives a curve with $2\ell-1$ real elliptic nodes, and the choice $k=\ell$ gives a curve with $2\ell-2$ imaginary nodes and one real hyperbolic node.
    Thus their Welschinger signs cancel each other.
\end{proof}

\begin{lemma}\label{lemma:odd_in_real_modification_count}
    Let $(\Gamma, h, \bs{\abstroppoint}, \{\varphi_V\}_{V\in \Gamma^0})$ be a conjugation invariant enhanced tropical curve and let $E\in \Gamma^{(1)}$ be a bounded edge of odd weight fixed by $\bc$.
    Then there exists a unique equivalence class of real modification data associated with $E$, and the number of elliptic nodes in it is even.
\end{lemma}

\begin{proof}
    Consider first the case of an unmarked edge, and denote its weight by $\ell$.
    Similarly to the proof of Lemma \ref{lemma:even_in_real_modification_vanish}, we can see that the defining equation of the modification data associated to $E$ is of the form $b_{-1}x^{-1} + 2\sqrt{b_1b_{-1}}T_{\ell}\left(\zeta_{\ell}^k \frac{y}{2\sqrt[\ell]{2\sqrt{b_1b_{-1}}}}\right)+b_1x$ for some $k=0,1,\ldots, \ell-1$.
    The only real solution is given by $k=0$, and it has $\ell-1$ real elliptic nodes.

    Now consider the case of marked edge.
    Out of the solutions considered in the proof of Lemma \ref{lemma:number_of_modifications} for the equation defining $C_{1,mod}$ only one solution is real.
    The curves $C_{1,mod}$ and $C_{2,mod}$ are smooth, so the number of elliptic nodes in the modification data associated with $E$ equals the number of elliptic nodes in the modification data associated with the unmarked edge of weight $\ell$ occurring in the modification data associated with $E$.
\end{proof}

\begin{lemma}\label{lemma:modification_conjugate_pair_edges}
    Let $(\Gamma, h, \bs{\abstroppoint}, \{\varphi_V\}_{V\in \Gamma^0})$ be a conjugation invariant enhanced tropical curve and let $E_1, E_2$ be a pair of edges of $\Gamma$ each of weight $\ell$ that are interchanged by $\bc$.
    Then there are $\ell$ equivalence classes of conjugation-invariant modification data associated with $E_1$ and $E_2$, and they have no real nodes.
\end{lemma}

\begin{proof}
    For one of the edges  we can pick any solution for the equation defining the modification data as in Lemma \ref{lemma:even_in_real_modification_vanish}.
    The equation for the other edge is the conjugate of the equation for the first edge.
    Both curves have no real nodes.
\end{proof}

\begin{theorem}\label{thm:real_final_count}
    Let $\bs{\embalgpoint}$ be essentially peripheral real point conditions satisfying the assumptions of Lemma \ref{lemma:real_tropicalization}.
    Let $\Gamma'$ be a regular trivalent plane tropical curve of genus $g$ that satisfies the reduced tropicalization of $\bs{\embalgpoint}$.
    Denote by $\Gamma'_{\text{re}}$ the subgraph of $\Gamma'$ containing all odd edges and their endpoints, and by $\Gamma'_{\text{im}}$ the subgraph of $\Gamma'$ containing all even edges and their endpoints.
    Assume that $\Gamma'_{\text{re}}\ne \emptyset$.
    Then the signed count of algebraic curves passing through $\bp$ and tropicalizing to some $\Gamma$ whose natural quotient is $\Gamma'$ is zero if $\Gamma'_{\text{im}}$ has any connected component $K$ with $\chi(K) < 1$ or $|K\cap \Gamma'_{\text{re}}| > 1$, and equal to 
    \begin{equation}\label{eq:real_conj_points_num}
        \prod_{V\in \Gamma'^{0}_{\text{re}}}(-1)^{\Int(V)}\prod_{V\in \Gamma'^{0}_{\text{im}}}\frac{\mu(V)}{2}\prod_{V\in \Gamma'^{0}_{\text{im}}\setminus\Gamma'^{0}_{\text{re}}}(-1)^{\mu(V)/4}
    \end{equation}
    otherwise.
\end{theorem}

\begin{proof}
    Let $\Gamma$ be a plane tropical curve of genus $g$ whose natural quotient is $\Gamma'$.
    If $\Gamma$ has an edge of even weight fixed by $\bc$ then the sum of Welschinger signs of all conjugation invariant modification data for $(\Gamma, h, \bs{\abstroppoint}, \{\varphi_V\}_{V\in \Gamma^0})$ vanishes by Lemma \ref{lemma:even_in_real_modification_vanish}.
    Otherwise, every even edge of $\Gamma'$ has $2$ preimages in $\Gamma$, so we have
    \begin{align*}
        1-g = \chi(\Gamma) & = \chi(\Gamma'_{\text{re}}) + 2\chi(\Gamma'_{\text{im}}) - 2|\Gamma'_{\text{re}} \cap \Gamma'_{\text{im}}| = \\
        & = \chi(\Gamma'_{\text{re}}) + 2\sum_{K\in \pi_0(\Gamma'_{\text{im}})} (\chi(K) - |K\cap \Gamma'_{\text{re}}|).
    \end{align*}
    On the other hand, 
    \begin{align*}
        1 - g = \chi(\Gamma') & = \chi(\Gamma'_{\text{re}}) + \chi(\Gamma'_{\text{im}}) - |\Gamma'_{\text{re}} \cap \Gamma'_{\text{im}}| = \\
        & = \chi(\Gamma'_{\text{re}}) + \sum_{K\in \pi_0(\Gamma'_{\text{im}})} (\chi(K) - |K\cap \Gamma'_{\text{re}}|).
    \end{align*}
    So $\sum_{K\in \pi_0(\Gamma'_{\text{im}})} (\chi(K) - |K\cap \Gamma'_{\text{re}}|)=0$ and since $\chi(K)\le 1$ and $|K\cap \Gamma'_{\text{re}}|\ge 1$ for all connected components $K$ of $\Gamma'_{\text{im}}$, we have $\chi(K)=1$ and $|K\cap \Gamma'_{\text{re}}|=1$ for all connected components $K$ of $\Gamma'_{\text{im}}$.

    For each vertex $V\in \Gamma'^0_{\text{im}}\setminus\Gamma'^0_{\text{re}}$ we have $2$ choices for how to connect the edges of $\Gamma$ that are adjacent to the preimage of $V$ in $\Gamma$, so there are $2^{|\Gamma'^0_{\text{im}}|\setminus|\Gamma'^0_{\text{re}}|}$ curves $\Gamma$ whose natural quotient is $\Gamma'$.
    By Lemmas \ref{lemma:real_enhancements_count}, \ref{lemma:odd_in_real_modification_count}, and \ref{lemma:modification_conjugate_pair_edges}, the signed count of conjugation invariant modification data for $\Gamma$ is given by
    \begin{align*}
        \prod_{V\in \Gamma'^{0}_{\text{im}}} \mu(V)\prod_{E\in \Gamma'^{1}_{\text{im}}} \frac{1}{2}\prod_{\Gamma'^{0}_{\text{im}}\setminus\Gamma'^{0}_{\text{re}}} (-1)^{\mu(V)/4}\prod_{V\in \Gamma'^{0}_{\text{re}}} (-1)^{\Int(V)}.
    \end{align*}
    Multiplying this by $2^{|\Gamma'^0_{\text{im}}\setminus\Gamma'^0_{\text{re}}|}$ and remembering that every vertex in $\Gamma'^0_{\text{im}}\setminus\Gamma'^0_{\text{re}}$ has $2$ incoming even edges while every vertex in $\Gamma'^0_{\text{im}}\cap\Gamma'^0_{\text{re}}$ has $1$ incoming even edge, we get exactly \eqref{eq:real_conj_points_num}.
    
    Finally, following the proof of Lemma \ref{lemma:complex_patchworking} and of \cite[Theorem 2.15]{shustin2006patchworking} we see that the types of nodes appearing in the resulting curve $[\bn: (\widehat C,\bs{\absalgpoint})\to\Tor_\K(P)]$ are exactly the types of nodes appearing in the modification data for $\Gamma$ so the Welschinger sign of the resulting curve is given by the Welschinger sign of the modification data for $\Gamma$.    
\end{proof}

\section{Relative refined tropical invariants}\label{sec:refined}
\subsection{Definition and invariance}

\begin{definition}
    For $a\in \mbb{Z}$ we denote
    \[ [a]_y^{-} = \frac{y^{a/2}-y^{-a/2}}{y^{1/2}-y^{-1/2}} \]
    where $y$ is a formal variable \footnote{We use here the notation $y$ instead of $q$ to avoid confusion with the points $\bs{\embtroppoint}\subset \R^2$.}.
    For a trivalent tropical curve $\Gamma$ whose unfixed ends all have weight $1$, we define the relative refined weight of $\Gamma$ as
    \begin{equation*}
        \wt(\Gamma) = \prod_{V\in \Gamma^0}[\mu(V)]_y^{-}\prod_{E\in \Gamma^1_{\infty}}\frac{1}{[\wt(E)]_y^{-}}.
    \end{equation*}
\end{definition}

\begin{proposition}\label{prop:refined_weight_invariant}
    For a type of tropical point conditions $\langle u_i \rangle_{i=1}^n$ with $\sum_{i=1}^n \min\{ \Vert u_i \Vert,1\} = |\partial P| + g - 1$ the sum $\sum_{\Gamma\in\text{Ev}^{-1}(\bs{\embtroppoint})} \wt(\Gamma)$ does not depend on the choice of the points $\bs{\embtroppoint}\in \prod_{i=1}^n \nicefrac{\R^2}{u_i\R}$ as long as it is in general position.
\end{proposition}

\begin{proof}
    First note that, by Lemma \ref{lemma:Ev_generic}, the sum $\sum_{\Gamma\in\text{Ev}^{-1}(\bs{\embtroppoint})} \wt(\Gamma)$ is well-defined for all $\bs{\embtroppoint}\in \prod_{i=1}^n \nicefrac{\R^2}{u_i\R}$ in general position.
    Moreover, similarly to the proofs of \cite[Theorem 1]{itenberg2013block} and \cite[Theorem 5.6]{CNRBI}, it is enough to show that for an essentially enumerative type $\alpha$ of codimension $1$ we have $\sum_{\beta\in \mathcal{A}_{+}} \wt(\beta) = \sum_{\beta\in \mathcal{A}_{-}} \wt(\beta)$, where $\mathcal{A}_{+}$ and $\mathcal{A}_{-}$ are the sets of enumeratively essential regenerations $\beta$ of $\alpha$ with $\text{Ev}\left(\moduli{\beta} \right) \subseteq H_{+}$ and $\text{Ev}\left(\moduli{\beta} \right) \subseteq H_{-}$, respectively.

    By Lemma \ref{lemma:Ev_codim1} we need to consider three options for $\Gamma$, all of which were already considered in the proof of \cite[Theorem 1]{itenberg2013block}, thus the proposition follows.
\end{proof}

\begin{definition}\label{def:relative_refined_invariant}
    For a type of tropical point conditions $\langle u_i \rangle_{i=1}^n$ with $\sum_{i=1}^n \min\{ \Vert u_i \Vert,1\} = |\partial P| + g - 1$, define the relative refined invariant by
    \begin{equation*}
        RR_y(P, g, \langle u_i \rangle_{i=1}^n) = \sum_{\Gamma\in\text{Ev}^{-1}(\bs{\embtroppoint})} \wt(\Gamma)
    \end{equation*}
    where $\bs{\embtroppoint}\in \prod_{i=1}^n \nicefrac{\R^2}{u_i\R}$ is in general position.
\end{definition}

\begin{proposition}\label{prop:refined_invariant_complex}
    For $y\to 1$ the refined weight of $\Gamma$ tends to the number given in \eqref{eq:complex_tangency_num}.
\end{proposition}

\begin{proof}
    This is immediate from the fact that $ [a]_y^- \to a $ as $ y \to 1 $ and that all the unfixed ends of $\Gamma$ have weight $1$.
\end{proof}

\begin{proposition}\label{prop:refined_invariant_real}
    Let $\Gamma$ be a trivalent plane tropical curve of degree $\Delta$. Denote by $\Gamma_{\text{im}}$ the subgraph of $\Gamma$ containing all even edges, and by $\Gamma_{\text{re}}$ the subgraph of $\Gamma$ containing all odd edges.
    Assume that $\Gamma_{\text{re}}\ne \emptyset$.
    \begin{enumerate}
        \item The limit of the refined weight of $\Gamma$ at $y=-1$ is always defined and it does not vanish if, and only if, for all connected components $K$ of $\Gamma_{\text{im}}$ we have $\chi(K) =1$ and $|K\cap \Gamma_{\text{re}}| = 1$.
        \item If the limit does not vanish and all the ends of $\Gamma$ have weight at most $2$, then the limit is equal to \eqref{eq:real_conj_points_num}.
    \end{enumerate}
\end{proposition}

\begin{proof}
    
    Note that by Pick's theorem, for a trivalent vertex of odd multiplicity all adjacent edges have odd weight, and for a trivalent vertex of even multiplicity there are either one or three adjacent edges of even weight (there cannot be exactly two adjacent edges of even weight by the balancing condition).

    In the neighborhood of $y=-1$ for $a$ odd we have $ [a]^{-}_y = (-1)^{\frac{a-1}{2}}+o(1)$ and for $a$ even we have $[a]^{-}_y = (-1)^{\frac{a}{2}-1}\frac{a}{2}(y^{1/2}+y^{-1/2})+o(y^{1/2}+y^{-1/2})$.
    Thus, the limit of the refined weight at $y\to -1$ is defined whenever the number of even vertices is at least the number of ends of even weight, and this limit vanishes if and only if the number of even vertices is strictly larger than the number of ends of even weight.
    For a connected component $K$ of $\Gamma_{\text{im}}$, denote by $K_\infty$ the set of its ends, by $K_0$ the set of its vertices, and by $K_1$ the set of its edges.
    Note that $K\cap \Gamma_{\text{re}}\ne \emptyset$ and this intersection contains only univalent vertices of $K$.
    Thus we get 
    \begin{equation*}
        1 \ge \chi(K) = |K_0| - |K_1|,
    \end{equation*}
    and
    \begin{align*}
        2|K_1| = \sum_{V\in K_0} \val_K(V) & = |K_\infty| + |K\cap \Gamma_{\text{re}}| + 3\cdot |K_0\setminus ((K\cap \Gamma_{\text{re}})\cup K_\infty)| = \\
        & = 3|K_0| - 2|K_\infty| - 2|K\cap \Gamma_{\text{re}}|,
    \end{align*}
    where $\val_K(V)$ denotes the valency of $V$ as a vertex of $K$.
    Thus,
    \begin{align*}
        |K_0\setminus K_\infty| - |K_\infty| & = |K_0| - 2|K_\infty| = 2|K_1|-2|K_0| + 2|K\cap \Gamma_{\text{re}}| = \\
        & = -2\chi(K) + 2|K\cap \Gamma_{\text{re}}| \ge -2+2=0,
    \end{align*}
    meaning that the limit of the refined weight at $y\to -1$ is always defined and does not vanish if and only if, for all connected components $K$ of $\Gamma_{\text{im}}$, we have $\chi(K)=1$ and $|K\cap \Gamma_{\text{re}}|=1$.

    Now consider a tropical curve $\Gamma$ satisfying the above conditions and consider the regular orientation of its edges.
    For every even edge divide the contribution of its target vertex to the refined weight by $(-1)^{\wt(E)/2-1}[2]^{-}_y$ while multiplying the contribution of its source vertex by the same factor.
    Similarly, for every odd edge divide the contribution of its source vertex to the refined weight by $(-1)^{\frac{\wt(E)-1}{2}}$ while multiplying the contribution of its target vertex by the same factor.
    Then one sees, by Pick's theorem, that the limit of the refined weight at $y\to -1$ is equal to
    \begin{align}\label{eq:limit_of_refined_at_y_minus_1}
        \prod_{V\in \Gamma_{\text{im}}^{0}}\frac{\mu(V)}{2}\cdot \prod_{V\in \Gamma_{\text{im}}^{0}\setminus \Gamma_{\text{re}}^{0}}(-1)^{\frac{\mu(V)-|\Delta(V)|}{2}}\cdot \prod_{V\in \Gamma_{\text{re}}^{0}}(-1)^{\frac{\mu(V)-\Delta(V)}{2}-1} & = \nonumber \\
        = \prod_{V\in \Gamma_{\text{im}}^{0}}\frac{\mu(V)}{2}\cdot \prod_{V\in \Gamma_{\text{im}}^{0}\setminus \Gamma_{\text{re}}^{0}}(-1)^{\frac{\mu(V)-|\Delta(V)|}{2}}\cdot \prod_{V\in \Gamma_{\text{re}}^{0}}(-1)^{\Int(V)} &
    \end{align}
    Here, $\Delta(V)$ denotes the integer points on the perimeter of the dual triangle to $V$.
    For a vertex $V\in \Gamma_{\text{im}}^{0}\setminus\Gamma_{\text{re}}^{0}$ the dual triangle is double of some triangle $\delta_0$ and we get
    \begin{equation*}
        (-1)^{\mu(V)/2+|\Delta(V)|/2} = (-1)^{2\cdot Area(\delta_0) + |\Delta(\delta_0)|} = (-1)^{Area(\delta_0)} = (-1)^{\mu(V)/4}
    \end{equation*} 
    by Pick's theorem. Thus, \eqref{eq:limit_of_refined_at_y_minus_1} and \eqref{eq:real_conj_points_num} coincide.
\end{proof}

\subsection{Computation}\label{sec:computation}
Computing the refined invariants defined in the current work is already an interesting and nontrivial problem even for $P=\conv\{(0,0), (d,0), (0,d)\}$ (i.e., for the projective plane).
The difficulty is that it appears complicated (and perhaps even impossible) to arrange the point conditions so that all enumerated tropical curves are floor-decomposed.

In the present work, we restrict the computation to the case where all boundary points lie on the same boundary component, i.e., all nonzero direction vectors are parallel to one another.

\begin{proposition}
	Let $P=\conv\{(0,0), (d,0), (0,d)\}$ for some $d\in \mbb{Z}_{\ge 0}$.
	For a sequence of numbers $\lambda_i\in \mbb{Z}_{\ge 0}$ such that $\sum_{i=1}^n \lambda_i = 3d+g-1$ set $u_i=\lambda_i \cdot (-1,0)$ for all $i$.
	Then the relative refined invariant $RR_y(P, g, \langle u_i \rangle_{i=1}^n)$ can be computed using floor diagrams as in \cite{FominMikhalkin2010}.
\end{proposition}

\begin{proof}
	The tropical curves considered here are the same as those considered in \cite[Section 3.2]{FominMikhalkin2010} for computing relative Gromov-Witten invariants with $\rho=\left\langle 1^{d-\sum\lambda_i}\right\rangle$, so it remains only to specify the refined multiplicity of floor diagrams.
	It is immediate that the refined multiplicity of a floor diagram is given by
	\begin{equation*}
		\prod_{e} \left([w(e)]_y^{-}\right)^2.
	\end{equation*}
\end{proof}

\section{Negative result}\label{sec:negative_result}
The signed count of real curves can be performed for real point conditions that are not essentially peripheral, but in positive genus this count is not invariant, even under very restrictive assumptions.

\begin{definition}
    For a conjugation invariant set of points $\bs{\embalgpoint}\subset (\K^\times)^2$ we define the reduced tropicalization of $\bs{\embalgpoint}$ as the multiset of points $\trop(\bs{\embalgpoint})\subset \R^2$.
    The multiplicity of a point $\embtroppoint \in \trop(\bs{\embalgpoint})$ is the number of points in $\bs{\embalgpoint}$ tropicalized to $\troppoint$.
\end{definition}

\begin{proposition}\label{prop:non_peripheral_real_conj_correspondence}
    Let $n,s\in \mbb{Z}_{\ge 0}$ be such that $n+s = |\partial P| + g - 1$.
    Let $\bs{\embalgpoint}\subset (\K^\times)^2$ be a conjugation invariant set of points of size $n+s$ containing $s$ conjugate pairs of points and additional $n-s$ points fixed by conjugation.
    Assume that $\bs{\embalgpoint}$ is in general position among the configurations of conjugation invariant point conditions with $s$ pairs of conjugate points and additional $n-s$ points fixed by conjugation, and assume that its reduced tropicalization is in general position, and let $\Gamma'$ be a plane tropical curve satisfying the reduced tropicalization of $\bs{\embalgpoint}$.
    Denote by $\Gamma'_{\text{re}}$ the subgraph of $\Gamma'$ containing all odd edges and their endpoints, and by $\Gamma'_{\text{im}}$ the subgraph of $\Gamma'$ containing all even edges and their endpoints.
    Similarly, denote by $\bs{\embtroppoint}_{\text{re}}$ and $\bs{\embtroppoint}_{\text{im}}$ the subsets of $\bs{\embtroppoint}$ consisting of the points with multiplicity $1$ and $2$ respectively.
    Then the signed count of real curves passing through $\bs{\embalgpoint}$ and tropicalizing to some $\Gamma$ whose natural quotient is $\Gamma'$ vanishes unless all of the following conditions are met:
    \begin{enumerate}
        \item $\gen(\Gamma') + \sum _{K\in \pi_0(\Gamma'_{\text{im}})} (|K \cap \Gamma'_{\text{re}}| - \chi(K)) = g$ where the sum is taken over all connected components $K$ of $\Gamma'_{\text{im}}$.
        \item All the marked points in $\bs{\embtroppoint}_{\text{re}}$ are contained in the interior of edges of $\Gamma'_{\text{re}}$.
        \item All the marked points in $\bs{\embtroppoint}_{\text{im}}$ are contained either on vertices of $\Gamma'_{\text{re}}$ disjoint from $\Gamma'_{\text{im}}$ or in the interior of edges of $\Gamma'_{\text{im}}$.
        \item All the ends of $\Gamma'$ have weight either $1$ or $2$.
        \item $\Gamma'$ is trivalent and regular. It has no vertices with $2$ incoming odd edges and one outgoing even edge (w.r.t. the regular orientation) and no vertices with exactly $2$ adjacent even edges.
    \end{enumerate}

    In this case, the signed count is equal to
    \begin{equation}\label{eq:non_peripheral_real_conj_points_num}
        \prod_{V\in \Gamma'^0_{\text{re}}\cap \bs{\troppoint}}\mu(V)\prod_{V\in \Gamma_{\text{re}}^{0}}(-1)^{\Int(V)}\prod_{V\in \Gamma_{\text{im}}^{0}}\frac{\mu(V)}{2}\prod_{V\in \Gamma_{\text{im}}^{0}\setminus\Gamma_{\text{re}}^{0}}(-1)^{\mu(V)/4}
    \end{equation}
\end{proposition}

\begin{proof}
    By the argument at the end of the proof of Theorem \ref{thm:real_final_count}, the signed count in question is equal to the signed count of conjugation invariant modification data for $\Gamma$.
    By Lemma \ref{lemma:modification_conjugate_pair_edges} this signed count can differ from $0$ only if $\Gamma$ has no even edges fixed by $\bc$, so every even edge of $\Gamma'$ has $2$ preimages in $\Gamma$.
    Denote by $\Gamma_{\text{im}}$ the preimage of $\Gamma'_{\text{im}}$ in $\Gamma$ and by $\Gamma_{\text{re}}$ the preimage of $\Gamma'_{\text{re}}$ in $\Gamma$.
    Since $\gen(\Gamma) \le g$ we get, by a simple Euler characteristic argument, that 
    \begin{equation}\label{eq:non_peripheral_genus}
        \gen(\Gamma') + \sum _{K\in \pi_0(\Gamma'_{\text{im}})} (|K \cap \Gamma'_{\text{re}}| - \chi(K)) \le g
    \end{equation}
    where the sum is taken over all connected components $K$ of $\Gamma'_{\text{im}}$.

    Since vertices of $\Gamma_{\text{im}}$ correspond to non real components of an algebraic curve and edges of $\Gamma_{\text{im}}$ correspond to non real nodes, we get that $\bs{\embtroppoint}_{\text{re}}\cap \Gamma'_{\text{im}} = \emptyset$.
    Similarly, 
    \begin{equation}\label{eq:non_peripheral_im_re_intersection}
        \bs{\embtroppoint}_{\text{im}}\cap \Gamma'_{\text{re}} \subset \Gamma'^0.
    \end{equation}
    Denote by $\Delta'$ the degree of $\Gamma'$.
    Since every even end of $\Gamma'$ has weight at least $2$, we get that 
    \begin{equation}\label{eq:non_peripheral_degree}
        |\Delta'| \le |\partial P| - |\Gamma'_{\text{im}}\cap \Gamma'^0_{\infty}|.
    \end{equation}

    Let $K$ be a connected component of $\Gamma'_{\text{im}}$.
    Since $\bs{\embtroppoint}$ is in general position, we get that $K\setminus \bs{\abstroppoint}$ cannot have bounded components. Hence, for every connected component $L$ of $K\setminus \bs{\abstroppoint}$, we have $|L\cap \Gamma'_{\text{re}}| + |L\cap \Gamma'^0_{\infty}| \ge 1$.
    Summing over all connected components $L$ of $K\setminus \bs{\abstroppoint} $ we get that 
    \begin{equation}\label{eq:non_peripheral_degree_sum}
        |K \cap \Gamma'_{\text{re}}| + |K \cap \Gamma'^0_{\infty}| \ge \chi(K) + \sum_{\abstroppoint_i \in \bs{\abstroppoint}\cap K} (\val(\abstroppoint_i) - 1) \ge \chi(K) + |\bs{\abstroppoint}\cap K|.
    \end{equation}

    Finally, using the fact that $\bs{\embtroppoint}$ are in general position, we get that $\dim \openmoduli{[\Gamma']} \ge 2n $, so by Lemma \ref{lemma:dim_moduli_space} and all the above inequalities we get that (recall that $|\bs{\embtroppoint}_{\text{im}}| = s$, $|\bs{\embtroppoint}_{\text{re}}| = n-s$, and $n+s = |\partial P| + g - 1$)
    \begin{align*}
        2n \le & \dim \openmoduli{[\Gamma']} \le \\
        \le & |\Delta'| + \gen(\Gamma') - 1 + |\bs{\abstroppoint}\setminus \Gamma'^0| \le \\
        \overset{\text{\ref{eq:non_peripheral_degree}}}{\le} & |\partial P| - |\Gamma'_{\text{im}}\cap \Gamma'^0_{\infty}| + \gen(\Gamma') - 1 + |\bs{\abstroppoint}\setminus \Gamma'^0| \le \\
        \overset{\text{\ref{eq:non_peripheral_degree_sum}}}{\le} & |\partial P|  + \sum_{K\in \pi_0(\Gamma'_{\text{im}})} (|K \cap \Gamma'_{\text{re}}| - \chi(K)) - | \bs{\abstroppoint}\cap \Gamma'_{\text{im}}|  + \gen(\Gamma') - 1 + |\bs{\abstroppoint}\setminus \Gamma'^0| \le \\
        \overset{\text{\ref{eq:non_peripheral_genus}}}{\le} & |\partial P| + g - 1 - |\bs{\abstroppoint}\cap \Gamma'_{\text{im}}| + |\bs{\abstroppoint}\setminus \Gamma'^0| = \\
        = & n+s - |\bs{\abstroppoint}\cap \Gamma'_{\text{im}}| + |\bs{\abstroppoint}\setminus \Gamma'^0| = n+s - |\bs{\embtroppoint}_{\text{im}}\cap \Gamma'_{\text{im}}| + |\bs{\embtroppoint}_{\text{re}}\setminus \Gamma'^0| = \\
        = & n+s + |\bs{\embtroppoint}_{\text{im}}\setminus \Gamma'^0| - |\bs{\embtroppoint}_{\text{im}}\cap \Gamma'_{\text{im}}| + |\bs{\embtroppoint}_{\text{re}}\setminus \Gamma'^0| \le \\
        \overset{\text{\ref{eq:non_peripheral_im_re_intersection}}}{\le} & 2n
    \end{align*}
    which implies that all the above inequalities are equalities and thus all the vanishing conditions of the proposition.

    The proof of \eqref{eq:non_peripheral_real_conj_points_num} is identical to the proof of Theorem \ref{thm:real_final_count}, the only difference is the number of enhancements corresponding to a marked vertex of $\Gamma'_{\text{re}}$ which was computed to be $\mu(V)$ in \cite[Lemma 2.5]{Shustin_welschinger_trop06}.
\end{proof}

\begin{proposition}
    For genus $g=1$, degree $d=4$, and $s=1$ pair of conjugate points, given a conjugation invariant set of points $\bs{\embalgpoint}\subset (\K^\times)^2$ of size $3d+g-1=12$ containing $s=1$ pair of conjugate points and additional $10$ points fixed by conjugation such that the reduced tropicalization of $\bs{\embalgpoint}$ is in Mikhalkin position (see \cite[Section 4]{CNRBI}), the signed count of degree $4$ genus $1$ real curves in $\PP^2_\K$ passing through $\bs{\embalgpoint}$ is either $63$ or $69$ depending on which of the tropical points in the reduced tropicalization of $\bs{\embalgpoint}$ is the image of the conjugate pair of points.
\end{proposition}

\begin{proof}
    Since the points are in Mikhalkin position, the required signed count can be computed using either the lattice-path algorithm or floor diagrams.
    Computation using the former method following \cite[Section 4]{Shustin_welschinger_trop06} and \cite[Section 5.3]{Schroeter_shustin_refined_elliptic08} was implemented by the second author in \cite{sinichkin_tropical_welschiner_counter_example_2026}, where it can be seen that if the image of the conjugate pair of points is the sixth point from the left, the signed count is $63$, and if it is any other point, the signed count is $69$.
\end{proof}

\bibliographystyle{amsplain-doi}
\bibliography{bib_collection}
\end{document}